\documentclass{tran-l}
\usepackage{amssymb,amsmath,amscd,amsthm,xspace,tocvsec2,color}
\usepackage{mathrsfs,enumerate}
\usepackage[breaklinks=true]{hyperref}%To get clickable links to papers

\theoremstyle{plain}

\newtheorem{theorem}[equation]{Theorem}
\newtheorem{lemma}[equation]{Lemma}
\newtheorem{prop}[equation]{Proposition}
\newtheorem{cor}[equation]{Corollary}

\newtheorem{utheorem}{\textrm{\textbf{Theorem}}}

\theoremstyle{definition}
\newtheorem{defn}[equation]{Definition}

\newtheorem{remark}[equation]{\normalfont{\textit{Remark}}}

\numberwithin{equation}{section}

\newcommand{\avg}{\operatorname{avg}}
\newcommand{\wt}{\operatorname{wt}}
\newcommand{\conv}{\operatorname{conv}}

\newcommand{\supp}{\operatorname{supp}}
\newcommand{\spa}{\operatorname{span}}

\newcommand{\vla}{\ensuremath{\mathbb{V}^\lambda}}
\newcommand{\R}{\ensuremath{\mathbb{R}}}
\newcommand{\C}{\ensuremath{\mathbb{C}}}

\newcommand{\Q}{\ensuremath{\mathbb{Q}}}
\newcommand{\bba}{\ensuremath{\mathbb{A}}}
\newcommand{\scrf}{\ensuremath{\mathscr{F}}}

\newcommand{\Z}{\ensuremath{\mathbb{Z}}}
\newcommand{\calp}{\ensuremath{\mathcal{P}}\xspace}
\newcommand{\liehr}{\ensuremath{\mathfrak{h}_{\mathbb{R}}}}

\newcommand{\lie}[1]{\ensuremath{\mathfrak{#1}\xspace}}
\newcommand{\comment}[1]{}

\begin{document}
\title[Standard parabolic subsets of highest weight modules]{Standard
parabolic subsets\\ of highest weight modules}

\author{Apoorva Khare}
\email[A.~Khare]{\tt khare@stanford.edu}
\address{Departments of Mathematics and Statistics, Stanford University,
Stanford, CA 94305, USA}
\date{\today}

\subjclass[2010]{Primary: 17B10; Secondary: 17B20, 52B15, 52B20}
\keywords{Highest weight module, parabolic Verma module, Weyl polytope,
standard parabolic subset, face map, inclusion relations, $f$-polynomial}

\begin{abstract}
In this paper we study certain fundamental and distinguished subsets of
weights of an arbitrary highest weight module over a complex semisimple
Lie algebra.
These sets ${\rm wt}_J \mathbb{V}^\lambda$ are defined for each highest
weight module $\mathbb{V}^\lambda$ and each subset $J$ of simple roots;
we term them ``standard parabolic subsets of weights''.
It is shown that for any highest weight module, the sets of simple roots
whose corresponding standard parabolic subsets of weights are equal form
intervals in the poset of subsets of the set of simple roots under
containment. Moreover, we provide closed-form expressions for the maximum
and minimum elements of the aforementioned intervals for all highest
weight modules $\mathbb{V}^\lambda$ over semisimple Lie algebras
$\mathfrak{g}$. Surprisingly, these formulas only require the Dynkin
diagram of $\mathfrak{g}$ and the integrability data of
$\mathbb{V}^\lambda$. As a consequence, we extend classical work by
Satake, Borel--Tits, Vinberg, and Casselman, as well as recent variants
by Cellini--Marietti to all highest weight modules.

We further compute the dimension, stabilizer, and vertex set of standard
parabolic faces of highest weight modules, and show that they are
completely determined by the aforementioned closed-form expressions.
We also compute the $f$-polynomial and a minimal half-space
representation of the convex hull of the set of weights. These results
were recently shown for the adjoint representation of a simple Lie
algebra, but analogues remain unknown for any other finite- or
infinite-dimensional highest weight module. Our analysis is uniform and
type-free, across all semisimple Lie algebras and for arbitrary highest
weight modules.
\end{abstract}
\maketitle

\setcounter{page}{2363}

%{{{1 Section 1 - Introduction
\section{Introduction}

This paper continues the analysis of arbitrary highest weight modules
over a complex semisimple Lie algebra $\lie{g}$, which was initiated in
\cite{Khwf}. Highest weight modules are fundamental in the study of Lie
algebras and representation theory. Classically, they are crucial in
studying the (parabolic) Bernstein-Gelfand-Gelfand Category
$\mathcal{O}$, primitive ideals of $U(\lie{g})$, quantum groups and
crystals (see e.g.~\cite{Dix,HK,H3}). More recently, highest weight
modules and certain distinguished subsets of their weights have gained
renewed attention for several reasons, including the study of
Kirillov--Reshetikhin modules over quantum affine Lie algebras, abelian
ideals, categorification, weight multiplicities, and the combinatorics of
root and (pseudo-)Weyl polytopes.
(See e.g.~\cite{CM,CP,CG2,Ka,Ka2,Me}, and the references therein and also
in \cite{Khwf}.)
Yet while certain special families such as (parabolic) Verma modules and
their finite-dimensional quotients are well-understood (but not fully
so), the same is not true of arbitrary highest weight modules.

The goal of the present paper -- as of previous work \cite{Khwf} -- is to
improve our understanding of general highest weight modules $\vla$. (All
notation is explained in Section \ref{S2}.) In \cite{Khwf}, the notion of
the Weyl polytope was extended, from finite-dimensional simple modules to
general highest weight modules. More precisely, we showed that for a
large class of highest weight modules $\vla$ with highest weight $\lambda
\in \lie{h}^*$, the convex hull of the weights is a convex polyhedron
that is invariant under a distinguished parabolic subgroup of the Weyl
group.
Moreover, for certain highest weight modules $\vla$ -- including all
simple highest weight modules $L(\lambda)$ for $\lambda \in \lie{h}^*$ --
we computed the set of weights, in three different ways.

A question that has been studied in detail in the literature by Satake,
Borel--Tits, Vinberg, Casselman, Cellini--Marietti, and others (see
\cite[Section 2.4]{Khwf}), involves analyzing inclusion relations between
pairs of faces of the aforementioned polyhedron for finite-dimensional
highest weight modules, i.e., for the family of Weyl polytopes. However,
the redundancies in this analysis have not been fully classified. Similar
results for general highest weight modules remain even more elusive. In
the first main result of this paper (see Theorem \ref{T5} in Section
\ref{S3}), we show that to each ``standard parabolic face'' correspond
certain unique minimum and maximum sets of simple roots. Moreover, we
provide closed-form expressions for these extremal subsets of simple
roots, in the process completely classifying the aforementioned
redundancies, for arbitrary highest weight modules.

Note that neither the existence of these extremal subsets of roots, nor
closed-form formulas for them, were known even for the thoroughly studied
and well-understood family of simple, finite-dimensional
$\lie{g}$-modules, excepting the adjoint representation. In contrast, our
formulas hold for arbitrary highest weight modules and arbitrary faces,
and are type-free.
Moreover, these formulas can be deduced directly from the root data of
the Dynkin diagram of $\lie{g}$ and the set of ``integrable simple
roots'' for the highest weight module. In particular, the present paper
completely resolves, for all modules $\vla$, the aforementioned problem
studied by Vinberg \cite{Vi} and others for finite-dimensional $\vla$. We
also extend recent work \cite{CM}, which discussed the adjoint
representation for simple $\lie{g}$.

Another goal is to compute convexity-theoretic quantities related to
highest weight modules. For instance, it is natural to ask what is the
dimension, or the stabilizer subgroup in the Weyl group, of a given face
of the convex hull of the set of weights. Moreover, if this hull is a
convex polyhedron, is it possible to compute its $f$-polynomial, or a
minimal half-space representation? Once again, the answers are not known
in the literature except for the adjoint representation for simple
$\lie{g}$. Our second objective in this paper is to answer all of these
questions for general highest weight modules $\vla$.
We show in this paper that not only Weyl polytopes/finite-dimensional
simple modules, but in fact all highest weight modules over semisimple
$\lie{g}$, are amenable to detailed analysis via representation theory
and convex analysis. See Theorem \ref{Tfinite} in Section \ref{S3} for
more details.

A related motivation involves the recent interesting work \cite{CM} of
Cellini and Marietti. In \cite{CM}, the authors studied the combinatorics
of irreducible root systems in great detail, and established several
interesting properties of the root polytope in terms of the corresponding
affine root system. (Recall that the root polytope is defined for a
simple Lie algebra $\lie{g}$ as the convex hull of the set of roots, and
it coincides with the Weyl polytope for the highest root.)
It is natural to ask if the results in \cite{CM} are specific
manifestations of broader phenomena that hold for all Weyl polytopes, or
even for the convex hulls of weights of arbitrary highest weight modules
$\vla$ (which naturally extend the notion of the Weyl polytope, by
results in \cite{Khwf}). Formulating and proving such results would
result in a deeper understanding of Weyl polytopes, finite-dimensional
$\lie{g}$-modules, and more.
However, we were unable to find such results in the literature other than
for the adjoint representation for simple $\lie{g}$ (in \cite{CM}).
Indeed, the situation for semisimple $\lie{g}$ and general $\lambda,
\vla$ is far more involved than that for simple $\lie{g}$ and the adjoint
representation, or even for finite-dimensional modules. There are various
technical reasons why the case of general $\vla$ is much harder, as we
explain below (see Remark \ref{Rminmax}). Nevertheless, the present paper
completely settles many of the questions addressed in \cite{CM}, by
proving them for every module $\vla$.

A special case of our main theorems is that Cellini--Marietti's results
hold with the exact same formulas, for all finite-dimensional $\vla$ with
$\lambda$ having the same support as the affine root of (simple)
$\mathfrak{g}$. Thus we immediately obtain hitherto unknown results about
a large family of Weyl polytopes, with proofs coming from representation
theory as opposed to the combinatorics of root systems.

\subsection*{Organization}

This paper is organized as follows. Section \ref{S2} reviews basic
notation as well as previous results in the literature, which lead to
several natural and motivating questions. We then present our three main
results in Section \ref{S3}. Each of the following three sections is
devoted to proving one of these main theorems.
These results discuss inclusion relations between standard parabolic
subsets of weights of highest weight modules, as well as computing
dimensions of affine hulls, stabilizers in parabolic Weyl subgroups,
$f$-polynomials, and minimal half-space representations in arbitrary
highest weight $\lie{g}$-modules.
In Section \ref{S7}, we show that the long roots, as well as the highest
and lowest roots among them, which are standard notions in the adjoint
representation, have natural analogues in compact standard parabolic
faces of all highest weight modules. Finally in Section \ref{Scellini},
we provide a dictionary to explain how our work specializes to the
results in recent work \cite{CM} on the root polytope for a simple Lie
algebra, in terms of the corresponding affine root system. We further
show how the results in \cite{CM} in fact hold for a large family of Weyl
polytopes.
%}}}

%{{{1 Section 2 - Notation and preliminary results
\section{Notation and preliminary results}\label{S2}

We begin by setting some notation. Given an $\R$-vector space
$\mathbb{V}$ and $R \subset \R$, $X,Y \subset \mathbb{V}$, define:
(a) $R_+ := R \cap [0,\infty)$,
(b) $X \pm Y$ to be the Minkowski sum and difference of $X,Y$,
(c) $\conv_\R X$ to denote the convex hull of $X$; and
(d) let $RX$ denote the set of all finite $\R$-linear combinations of
elements of $X$ with coefficients in $R$.

Fix a complex semisimple Lie algebra $\lie{g}$ as well as a triangular
decomposition $\lie g = \lie n^+ \oplus \lie h \oplus \lie n^-$. Let the
corresponding root system be $\Phi$. Corresponding to $\lie{n}^+$, fix
distinguished $\C$-bases of simple roots $\Delta := \{ \alpha_i : i \in I
\} \subset \Phi$ and the associated fundamental weights $\Omega := \{
\omega_i : i \in I \}$, both indexed by $I$. For any $J \subset I$,
define $\Delta_J := \{ \alpha_j : j \in J \}$, and $\Omega_J$ similarly.
Let $\liehr^* := \R \Delta$ be the real form of $\lie{h}^*$; then
$\liehr^* = \R \Omega = \R \Omega_I$.
 Moreover, $\lie{h}^*$ has a standard partial order via: $\lambda \geq
\mu$ if $\lambda - \mu \in \Z_+ \Delta$.
Next, define $P := \Z \Omega \supset Q := \Z \Delta$ to be the weight and
root lattices in $\liehr^*$ respectively, and
\begin{equation}
\begin{aligned}
P^+_J := \Z_+ \Omega_J, \quad Q^+_J := \Z_+ \Delta_J, &
\quad P^+ := P^+_I, \quad Q^+ := Q^+_I,\\
\Phi^\pm_J := \Phi \cap \pm Q^+_J, & \quad \Phi^\pm := \Phi^\pm_I.
\end{aligned}
\end{equation}

\noindent Thus, $P^+ = P^+_I$ is the set of dominant integral weights.
Let $(,)$ be the positive definite symmetric bilinear form on $\liehr^*$
induced by the restriction of the Killing form on $\lie g$ to $\liehr$.
Then $(\omega_i, \alpha_j) = \delta_{i,j} (\alpha_j, \alpha_j) / 2\
\forall i,j \in I$.
Also define $h_i$ to be the unique element of $\lie{h}$ identified with
$(2 / (\alpha_i, \alpha_i)) \alpha_i$ via the Killing form. The elements
$h_i$ form an $\R$-basis of $\lie{h}_\R$. Fix a set of Chevalley
generators $\{ x_{\alpha_i}^\pm \in \lie n^\pm : i \in I \}$ such that
$[x_{\alpha_i}^+, x_{\alpha_j}^-] = \delta_{ij} h_i$ for all $i,j \in I$.
Also extend $(,)$ to all of $\lie{h}^*$. Finally, the Weyl group is the
finite subgroup $W \subset O(\lie{h}^*)$ generated by the simple
reflections $\{ s_i = s_{\alpha_i} : i \in I \}$, where $s_i$ sends a
weight $\lambda \in \lie{h}^*$ to $\lambda - \lambda(h_i) \alpha_i$. Now
define $W_J$ to be the subgroup of $W$ generated by $\{ s_j : j \in J
\}$, with unique longest element $w_\circ^J$. 

We now discuss various distinguished highest weight modules. Given
$\lambda \in \lie{h}^*$, define $M(\lambda)$ to be the Verma module with
highest weight $\lambda$, and $L(\lambda)$ to be its unique simple
quotient. Thus $M(\lambda) := U\lie{g} / U \lie{g}(\lie{n}^+ + \ker
\lambda)$. A highest weight module is a quotient of some Verma module,
and is usually denoted in this paper by $M(\lambda) \twoheadrightarrow
\vla$. Note that $\vla$ is finite-dimensional if and only if $\lambda \in
P^+$ is dominant integral and $\vla = L(\lambda)$ is simple. In this case
the convex hull of the weights $\wt L(\lambda)$ is a compact polytope,
called a \textit{Weyl polytope} and denoted by $\calp(\lambda) :=
\conv_\R \wt L(\lambda)$.

Given $\lambda \in \lie{h}^*$, define $J_\lambda := \{ i \in I :
\lambda(h_i) \in \Z_+ \}$. Let $\lie{g}_J$ denote the semisimple Lie
subalgebra of $\lie{g}$ generated by $\{ x_{\alpha_j}^\pm : j \in J \}$,
and define the parabolic Lie subalgebra $\lie{p}_J := \lie{g}_J + \lie{h}
+ \lie{n}^+$ for all $J \subset I$. Now given $\lambda \in \lie{h}^*$ and
$J \subset J_\lambda$, define the ($J$-)parabolic Verma module with
highest weight $\lambda$ to be $M(\lambda,J) := U(\lie{g})
\otimes_{U(\lie{p}_J)} L_J(\lambda)$. Here, $L_J(\lambda)$ is a simple
finite-dimensional highest weight module over the Levi subalgebra
$\lie{h} + \lie{g}_J$; it is also killed by $\lie{g}_{I \setminus J} \cap
\lie{n}^+$ (in $M(\lambda, J)$). Note that parabolic Verma modules
include all Verma modules (for which $J = \emptyset$) and all
finite-dimensional simple modules (for which $\lambda \in P^+$ and $J =
I$).\medskip

The following definition introduces the principal objects of study in the
paper. As we presently discuss, these objects have been studied in the
literature for finite-dimensional $\lie{g}$-modules, most intensively for
the adjoint representation.

\begin{defn}
Given $\lambda \in \lie{h}^*$ and $M(\lambda) \twoheadrightarrow \vla$, a
\textit{standard parabolic (sub)set of weights of $\wt \vla$} is $\wt_J
\vla := (\lambda - \Z_+ \Delta_J) \cap \wt \vla$ for some $J \subset I$;
and a \textit{standard parabolic face} is $\conv_\R \wt_J \vla$.
\end{defn}

We now recall previous work involving standard parabolic subsets and
faces for various classes of highest weight modules $M(\lambda)
\twoheadrightarrow \vla$. To our knowledge, these results were previously
known for parabolic Verma modules, but not for other highest weight
modules. The first result establishes integrability for a unique
distinguished subset of simple roots, for arbitrary modules $\vla$.

\begin{theorem}[Khare, {\cite[Theorem A]{Khwf}}]\label{T1}
Given $\lambda \in \lie{h}^*$ and $M(\lambda) \twoheadrightarrow \vla$,
there exists a unique subset $J(\vla) \subset I$ such that the following
are equivalent:
(a) $J \subset J(\vla)$;
(b) $\wt_J \vla$ is finite; 
(c) $\wt_J \vla$ is $W_J$-stable;
(d) $\wt \vla$ is $W_J$-stable.
Moreover, if $\vla_\lambda$ is spanned by $v_\lambda$, then
\begin{equation}
J(\vla) := \{ i \in J_\lambda : (x^-_{\alpha_i})^{\lambda(h_i) + 1}
v_\lambda = 0 \}.
\end{equation}

\noindent In particular, if $\vla$ is a parabolic Verma module
$M(\lambda,J')$ for $J' \subset J_\lambda$, or a simple module
$L(\lambda)$, then $J(\vla) = J'$ or $J_\lambda$ respectively. 
\end{theorem}

The subset $J(\vla)$ is ubiquitous and crucially used throughout the
remainder of the paper.

The next result establishes for a large class of highest weight modules
$\vla$ that the convex hull of the set of weights is a polyhedron. It is
also possible to compute the vertices, faces, and stabilizer subgroup in
$W$ of this polyhedron.

\begin{theorem}[Khare, {\cite[Theorems B and C]{Khwf}}]\label{T23}
Suppose $(\lambda, \vla)$ satisfy one of the following:
(a) $\lambda(h_i) \neq 0\ \forall i \in I$ and $\vla$ is arbitrary;
(b) $|J_\lambda \setminus J(\vla)| \leq 1$ (e.g., if $\vla$ is simple for
any $\lambda \in \lie{h}^*$);
(c) $\vla = M(\lambda, J')$ for some $J' \subset J_\lambda$; or
(d) $\vla$ is pure (in the sense of \cite{Fe}).

Then the convex hull (in Euclidean space) $\conv_\R \wt \vla \subset
\lambda + \liehr^*$ is a convex polyhedron with vertices
$W_{J(\vla)}(\lambda)$, and the stabilizer subgroup in $W$ of both $\wt
\vla$ and $\conv_\R \wt \vla$ is $W_{J(\vla)}$. Moreover, a nonempty
subset $Y \subset \wt \vla$ maximizes a linear functional $\varphi \in
\lie{h}$ (i.e., $Y$ is the set of weights on a supporting hyperplane)
if and only if $Y = w(\wt_J \vla)$ for some $w \in W_{J(\vla)}$ and $J
\subset I$.
\end{theorem}

Theorem \ref{T23} extends the notion of the Weyl polytope from
finite-dimensional modules $L(\lambda)$ to general highest weight modules
$\vla$. The result also extends the classification by Chari \textit{et
al} \cite{CDR} of all maximizer subsets in the set of roots (i.e., in
$\wt \lie{g}$), as well as previous work \cite{KR} on maximizer subsets
in the weights of parabolic Verma modules. Such results were used by
Chari and her coauthors in \cite{CG2,CKR} to study distinguished
categories and associated families of Koszul algebras arising from
Kirillov--Reshetikhin modules over quantum affine Lie algebras.\medskip

We now discuss the motivations behind the present paper.
Theorem \ref{T23} and previous work in \cite{Vi,KR} shows that the sets
$\wt_J \vla$ form a distinguished family of subsets of weights for a very
large class of highest weight modules $\vla$, including all parabolic
Verma modules $M(\lambda,J')$ and their simple quotients.
In particular, the standard parabolic subsets $\wt_J \vla$ and their
$W_{J(\vla)}$-conjugates are precisely the sets of weights on the faces
of the convex hull $\conv_\R \wt \vla$. Equivalently, standard parabolic
subsets and their Weyl group conjugates are precisely the maximizer
subsets of weights in $\wt \vla$. More precisely, if $\vla_\lambda = \C
v_\lambda$ and $\rho_J := \sum_{j \in J} \omega_j$, then
\begin{equation}\label{Eface}
\wt_J \vla = \wt \vla \cap (\lambda - \Z \Delta_J) = \wt U(\lie{g}_J)
v_\lambda = \arg \max_{\wt \vla} (\rho_{I \setminus J}, -).
\end{equation}

\noindent Also note by \cite[Theorem C]{Khwf} that the sets $\wt_J \vla$
are precisely the \textit{weak $\bba$-faces} of $\wt \vla$ for ``most''
$\vla$ and all additive subgroups $\bba \subset (\R,+)$. Thus, the goal
of this paper is to study standard parabolic sets of weights in detail,
for arbitrary highest weight $\lie{g}$-modules $\vla$.

Given a finite-dimensional simple module, or more generally a parabolic
Verma module, it is natural to classify its faces and the redundancies
among them. Results along these lines were shown by Satake, Borel--Tits,
Vinberg, and Casselman for Weyl polytopes $\conv_\R \wt M(\lambda,I)$;
see \cite[Section 2.4]{Khwf} for more details. Here we state Vinberg's
formulation using the present notation.

\begin{theorem}[Vinberg, \cite{Vi}]\label{Tvin}
Suppose $\lambda \in P^+$. Then $\conv_\R \wt L(\lambda)$ is a
$W$-invariant convex polytope with vertex set $W(\lambda)$. Every face of
this polytope is $W$-conjugate to a unique standard parabolic face
$\conv_\R \wt_J L(\lambda)$.
\end{theorem}

The property of interest here is the ``uniqueness'' of the standard
parabolic face (since the remainder of the result is already known for
general modules $\vla$ by Theorem \ref{T23}).

In an interesting recent paper \cite{CM}, Cellini and Marietti studied in
great detail, the \textit{root polytope} $\conv_\R \wt \lie{g}$ of a
simple Lie algebra $\lie{g}$ and the set of integral weights contained in
it. The authors showed several hitherto unexplored combinatorial
properties of $\wt \lie{g} = \Phi \sqcup \{ 0 \}$. In particular, they
classified redundancies between the faces of the root polytope $\conv_\R
\wt \lie{g}$:

\begin{theorem}[Cellini--Marietti, {\cite[Theorem 1.2]{CM}}]\label{Tcm1}
Suppose $\lie{g}$ is simple, with highest root $\theta \in \Phi^+$.
Given $\theta = \sum_{i \in I} m_i \alpha_i$ and $J \subset I$, define
$F_J \subset \R \Delta$ to be the convex hull of the roots $\mu \in \Phi$
such that $(\mu, \omega_j) = m_j$ for all $j \in J$. Then there exist
subsets $\partial J, \overline{J} \subset I$ such that for $J' \subset
I$,
\[ F_{J'} = F_J \quad \Longleftrightarrow \quad \partial J \subset J'
\subset \overline{J}. \]
\end{theorem}

In other words, the set of possible $J' \subset I$ such that $\wt_{J'}
\lie{g} = \wt_J \lie{g}$ forms an interval in the poset of subsets of $I$
under containment.
Note also that $F_J = \conv_\R \wt_{I \setminus J} \lie{g}$ for all $J
\subset I$. Thus Theorem \ref{Tcm1} says that for all $J$, there exist
smallest and largest sets $J_{\min}$ and $J_{\max}$, such that $\{ J' :
\conv_\R \wt_{J'} \lie{g} = \conv_\R \wt_J \lie{g} \} = [J_{\min},
J_{\max}]$.

The present discussion, including Theorems \ref{Tvin} and \ref{Tcm1}, can
be reformulated as follows. Using Theorem \ref{T1}, define the
\textit{face map} to be
\begin{equation}\label{Efacemap}
\scrf_{\vla} : W_{J(\vla)} \times 2^I \to 2^{\wt \vla}, \qquad
\scrf_{\vla}(w,J) := w(\wt_J \vla),
\end{equation}

\noindent where $2^S$ denotes the power set of a set $S$. We term
$\scrf_{\vla}$ the face map because by \eqref{Eface}, $\scrf_{\vla}(w,J)$
equals the set of weights on a face of $\conv_\R \wt \vla$ for all $w \in
W_{J(\vla)}$ and $J \subset I$. Next, define
\begin{equation}\label{Efacemap1}
\scrf^{(1)}_{\vla} = \scrf_{\vla}(1,-) : 2^I \to 2^{\wt \vla}, \qquad
\scrf^{(1)}_{\vla}(J) := \wt_J \vla.
\end{equation}

\noindent (Note that both $\scrf_{\vla}$ and $\scrf^{(1)}_{\vla}$ have
finite sets as domains.) 
Now Theorem \ref{Tvin} shows a connection between the fibers of
$\scrf_{L(\lambda)}$ and of $\scrf^{(1)}_{L(\lambda)}$: namely, if
$\lambda \in P^+$ and $\scrf_{L(\lambda)}(w,J) =
\scrf_{L(\lambda)}(w',J')$, then $\scrf^{(1)}_{L(\lambda)}(J) =
\scrf^{(1)}_{L(\lambda)}(J')$. On the other hand, Theorem \ref{Tcm1}
classifies the fibers of $\scrf^{(1)}_{\lie{g}}$ and shows that they are
intervals of the form $[J_{\min}, J_{\max}]$.
Note that a similar result holds trivially for arbitrary Verma modules
$\vla = M(\lambda)$, since the assignment $J \mapsto \wt_J M(\lambda)$ is
one-to-one. In previous work \cite{Khwf} we also showed a similar result
for all highest weight modules with ``generic'' highest weight:

\begin{theorem}[Khare, {\cite[Corollary 3.16]{Khwf}}]\label{Tsimplyreg}
Suppose $\lie{g}$ is semisimple, $\lambda(h_i) \neq 0\ \forall i \in I$,
and $M(\lambda) \twoheadrightarrow \vla$ is arbitrary. Then $J \mapsto
\wt_J \vla$ is an injective map from subsets of $I$ to subsets of $\wt
\vla$.
\end{theorem}

In other words, if $\lambda$ is not on any simple root hyperplane, then
$\scrf^{(1)}_{\vla}$ is injective on $2^I$.
(Thus it has fiber $[J,J]$ at each subset $J$.)

The following questions naturally arise now:
(1) In what level of generality do the above results hold? In other
words, for which highest weight modules $\vla$ can the fibers of
$\scrf_{\vla}$ and $\scrf^{(1)}_{\vla}$ be related and/or classified?
(2) If the analogue of Theorem \ref{Tcm1} holds for a highest weight
module $\vla$ (i.e., the fibers of $\scrf^{(1)}_{\vla}$ are all
intervals), then is there a closed-form expression for the sets
$J_{\min}, J_{\max}$?
(3) Can the redundancies among the set of faces $\{ w(\wt_J \vla) : w \in
W_{J(\vla)}, J \subset I \}$ -- i.e., the fibers of the map
$\scrf_{\vla}$ -- be classified for all highest weight modules $\vla$?
Answers to the second question are known only for the adjoint
representation for simple $\lie{g}$ (and now when $\lambda$ avoids the
simple root hyperplanes, by our result, Theorem \ref{Tsimplyreg}).
Neither the existence of $J_{\min}, J_{\max}$, nor formulas for them, are
known for other finite-dimensional modules $M(\lambda,I) = L(\lambda)$.
Formulas in the case of general modules $\vla$ are even harder, as are
the first and third questions.

Thus, one of the main accomplishments of this paper is to provide
positive answers to all of the above questions in complete generality,
for all highest weight modules $\vla$ and all semisimple $\lie{g}$.
Remarkably, the formulas for $J_{\min}$ and $J_{\max}$ can be read off
from the Dynkin diagram of $\lie{g}$, using only the data of the sets $J$
and $J(\vla)$ as well as the support set of $\lambda$. See Theorem
\ref{T5} for a precise formulation.
We also completely classify the redundancies between faces of arbitrary
highest weight modules $\vla$; i.e., we compute the fibers of the face
map $\scrf_{\vla}$. See Proposition \ref{Pvinberg}.

Our second motivation comes from further combinatorial results shown
recently by Cellini and Marietti \cite{CM} for the root system of a
simple Lie algebra $\lie{g}$. In \cite{CM} the authors compute the
dimension and stabilizer subgroup of the faces of the root polytope of
$\lie{g}$.

\begin{theorem}[Cellini--Marietti, {\cite[Section 1]{CM}}]\label{Tcm2}
Suppose $\lie{g}$ is simple, with highest root $\theta$. Given $\lambda =
\theta$ and $J \subset I$, define $\partial J, \overline{J}, F_J$ as in
Theorem \ref{Tcm1}. Then $F_J$ has dimension $|I| - |\overline{J}|$ and
stabilizer subgroup $W_{I \setminus \partial J}$ in $W$, and its
barycenter lies in $\R_+ \Omega_{\partial J}$. Moreover, $F_J$ has $[W_{I
\setminus J} : W_{(I \setminus J) \cap \{ \theta \}^\perp}]$ vertices,
where $\{ \theta \}^\perp$ denotes the set of simple roots orthogonal to
$\theta$.
\end{theorem}

It is natural to ask if similar results hold for Weyl polytopes, or more
generally for all modules $\vla$. The present paper provides positive
answers for general highest weight modules $\vla$ over all semisimple
$\lie{g}$; see Theorem \ref{Tfinite} as well as Lemma \ref{Lfinite}.
%}}}

%{{{1 Section 3 - Main results
\section{Main results}\label{S3}

In this section we present the main results of the paper. The following 
notation is required to state and prove these results.

\begin{defn}\label{Dminmax}
Suppose $\lambda \in \lie{h}^*$, $M(\lambda) \twoheadrightarrow \vla$,
and $J \subset I$.
\begin{enumerate}
\item Define $\pi_J : \lie{h}^* = \C \Omega_I \twoheadrightarrow \C
\Omega_J$ to be the projection map with kernel $\C \Omega_{I \setminus
J}$.

\item Denote the support of $\lambda \in \lie{h}^*$ by $\supp(\lambda) :=
\{ i \in I : (\lambda, \alpha_i) \neq 0 \} = \{ i \in I : \pi_{\{ i
\}}(\lambda) \neq 0 \}$. Thus $\lambda \in \C \Omega_{\supp(\lambda)}\
\forall \lambda \in \lie{h}^*$.

\item If $J \subset J(\vla)$, define $J_{\min}$ to be the union of the
connected components $C$ in the Dynkin diagram of $J$ such that
$\pi_C(\lambda) \neq 0$, i.e., such that $C \cap \supp(\lambda)$ is
nonempty.

\item Given $J \subset I$ and $\vla$, partition $I$ as
$I = \bigsqcup_{i=1}^6 J_i(\vla)$, where
\[ J_1(\vla) := I \setminus (J \cup J(\vla)), \qquad
J_2(\vla) := J \setminus J(\vla), \qquad
J_3(\vla) := (J \cap J(\vla))_{\min}, \]
\[ J_5(\vla) := (J \cap J(\vla)) \setminus
(J_3(\vla) \sqcup J_4(\vla)),
\qquad J_6(\vla) := J(\vla) \setminus J, \]

\noindent and $J_4(\vla)$ is the union of those connected components
$C$ of the Dynkin diagram of $J \cap J(\vla)$, for which $\pi_C(\lambda)
= 0$ but $\Delta_C \not\perp \Delta_{J_2(\vla)}$.

\item Define for any $X \subset \lie{h}^*$, the set $X^\perp := \{ i \in I
: (\alpha_i, x) = 0\ \forall x \in X \}$. Now define $J^\perp :=
\Delta_J^\perp$.
\end{enumerate}
\end{defn}

We provide some elaboration on the sets $J_i(\vla)$. These explanations
are helpful in building some intuition about the sets $J_i(\vla)$; all
reasoning is provided below, in the proof of Theorem \ref{T5}.
\begin{itemize}
\item $J_1(\vla)$ does not play any role in the set $\wt_J \vla$ for a
highest weight module $\vla$.
\item $J_2(\vla)$ consists of the indices $j \in J$ such that
$\lambda - \Z_+ \alpha_j \subset \wt_J \vla$.
\item For $l = 3,4,5$, $J_l(\vla)$ consists of the graph components $C$
of the Dynkin diagram of $J \cap J(\vla)$, for which:
(a) $\pi_C(\lambda) \neq 0$, if $l=3$;
(b) $\pi_C(\lambda) = 0$ but $(\Delta_{J \setminus J(\vla)}, \Delta_C)
\neq 0$, if $l=4$; and
(c) $\pi_C(\lambda) = 0 = (\Delta_{J \setminus J(\vla)}, \Delta_C)$, if
$l=5$.
\end{itemize}

We now state the main results of this paper. Our first result achieves
three goals: first, it provides a complete characterization of when two
standard parabolic sets of weights coincide, for all highest weight
modules $\vla$. (As explained in Remark \ref{Rcellini}, this allows for
the complete classification of inclusion relations between standard
parabolic subsets/faces.)
Next, the result also demonstrates the existence of the sets $J_{\min},
J_{\max}$ for all $\vla$. Finally, it provides the first closed-form
expressions for the sets $J_{\min}, J_{\max}$ for any highest weight
module over semisimple $\lie{g}$ (with the sole exception of $\vla =
\lie{g}$ for simple $\lie{g}$, for which a different pair of formulae
were proved very recently in \cite{CM}).

\begin{utheorem}\label{T5}
Suppose $\lambda \in \lie{h}^*$, $M(\lambda) \twoheadrightarrow \vla$,
$\vla_\lambda$ is spanned by $v_\lambda$, and $J \subset I$.
Define subsets $J_{\min}, J_{\max} \subset J(\vla)$ as follows (notation
as in Definition \ref{Dminmax}):
\begin{align}
J_{\min} = &\ J_3(\vla) \sqcup J_4(\vla)
= \bigcup_{\mu \in \{ 0 \} \sqcup \Delta_{J \setminus J(\vla)}} (J \cap
J(\vla))_3(L_{J(\vla)}(\pi_{J(\vla)}(\lambda - \mu))),\label{Emin}\\
J_{\max} = &\ (J \cap J(\vla)) \sqcup [ J_6(\vla) \cap \{ \lambda
\}^\perp \cap J_{\min}^\perp \cap J_2(\vla)^\perp]\label{Emax}\\
= &\ J_{\min} \sqcup (J(\vla) \cap \{ \lambda \}^\perp \cap
J_{\min}^\perp \cap J_2(\vla)^\perp).\notag
\end{align}

\noindent Then all equalities in \eqref{Emin} and \eqref{Emax} are valid.
Moreover, the following are equivalent for $J' \subset I$:
\begin{enumerate}
\item There exist $w,w' \in W_{J(\vla)}$ such that $w(\wt_J \vla) =
w'(\wt_{J'} \vla)$.

\item $\wt_J \vla = \wt_{J'} \vla$.

\item $\conv_\R \wt_J \vla = \conv_\R \wt_{J'} \vla$.

\item $U(\lie{g}_J) v_\lambda = U(\lie{g}_{J'}) v_\lambda$.

\item $J \setminus J(\vla) = J' \setminus J(\vla)$ and $J_{\min} \subset
J' \cap J(\vla) \subset J_{\max}$.
\end{enumerate}
\end{utheorem}

In particular, there are three formulas for $J_{\min}$ when $J \subset
J(\vla)$ (from Definition \ref{Dminmax} and equation \eqref{Emin}), and
we show in this paper that these formulas agree, as do the two proposed
formulas for $J_{\max}$.
Moreover, Theorems \ref{Tvin}, \ref{Tcm1}, \ref{Tsimplyreg} respectively
by Vinberg (and others), Cellini--Marietti, and us, are the special cases
of Theorem \ref{T5} that were previously known. The second of these
assertions holds because $F_J = \conv_\R \wt_{I \setminus J} \lie{g}$ for
simple $\lie{g}$, as explained in Section \ref{Scellini} below. The third
assertion holds because if $\lambda(h_i) \neq 0$ for all $i \in I$, then
$J_4(\vla) = J_5(\vla) = \{ \lambda \}^\perp = \emptyset$ for all $\vla$,
whence $J_{\min} = J_{\max} = J_3(\vla) = J \cap J(\vla)$.

Further note from Theorem \ref{T5} that the last expression in equation
\eqref{Emin} reduces the situation to the finite-dimensional setting of
the ``integrable top'' $L_{J(\vla)}(\lambda - \mu)$. In fact the set
corresponding to $\mu = 0$ in \eqref{Emin} is contained in the set
corresponding to $\mu = \alpha_{j_2}$ for every $j_2 \in J_2(\vla)$, but
it is nevertheless included in \eqref{Emin} (and also in the results and
proofs in the paper below) because $J_2(\vla)$ may be empty, e.g.~in the
case of finite-dimensional modules $\vla$.

\begin{remark}
As an aside, Theorem \ref{T5} answers the following branching-type
question: given a finite-dimensional simple $\lie{g}$-module
$L(\lambda)$, for which subsets $J \subset I$ is its restriction to
$\lie{g}_J$ still an irreducible $\lie{g}_J$-module? By Theorem \ref{T5},
the answer is: all $J$ containing $I_{\min}$. More generally for all
$\lambda \in \lie{h}^*$, the answer to the same question (for
$L(\lambda)$) is: all subsets $J$ containing $I_{\min} \sqcup (I
\setminus J_\lambda)$.
\end{remark}

\begin{remark}
It is preferable to work with arbitrary semisimple $\lie{g}$ rather than
simple Lie algebras. One reason is that when $J \subset I$ is allowed to
be arbitrary, disjoint connected graph components of $J \cap J(\vla)$
automatically force us to work with semisimple subalgebras of $\lie{g}$.
\end{remark}

Our next main result establishes several combinatorial facts about
standard parabolic sets of weights $\wt_J \vla$ for highest weight
modules $\vla$.

\begin{utheorem}\label{Tfinite}
Suppose $\lambda \in \lie{h}^*$, $M(\lambda) \twoheadrightarrow \vla$,
and $J \subset I$. Then the standard parabolic face $\conv_\R (\wt_J
\vla)$ has:
\begin{enumerate}
\item dimension $|J_{\min}| + |J \setminus J(\vla)|$ as well as affine
hull $\lambda - \R \Delta_{J_{\min} \sqcup (J \setminus J(\vla))}$ and

\item stabilizer subgroup $W_{J_{\max}} = W_{J_{\min}} \times W_{J_{\max}
\setminus J_{\min}}$ in $W_{J(\vla)}$.
\end{enumerate}

\noindent Now assume further that $\conv_\R \wt \vla = \conv_\R \wt
M(\lambda, J(\vla))$. Then the face $\conv_\R \wt_J \vla$ has $[W_{J \cap
J(\vla)} : W_{J \cap J(\vla) \cap \{ \lambda \}^\perp}]$ vertices.
Moreover, the polyhedron $\conv_\R \wt \vla$ has $f$-polynomial
\[ {\bf f}_{\vla}(t) = \sum_J [W_{J(\vla)} : W_{J_{\max}}] t^{|J_{\min}|
+ |J \setminus J(\vla)|}, \]

\noindent where we sum over the distinct elements $J$ in the multiset $\{
J_{\max} : J \subset I \}$, or in the multiset $\{ J_{\min} : J \subset I
\}$.
\end{utheorem}

\begin{remark}
Note from Theorem \ref{Tfinite} that much of the convexity-theoretic data
of a highest weight module $\vla$ is completely determined by the sets
$\{ J_{\min}, J_{\max} : J \subset I \}$. (This includes $I_{\max} =
J(\vla)$.)
\end{remark}

The assertions in Theorem \ref{Tfinite} were established in the special
case $\lambda = \theta$ and $\vla = M(\theta,I) = L(\theta) = \lie{g}$
for a simple Lie algebra $\lie{g}$, by Cellini and Marietti in their
recent paper \cite{CM} (see also \cite{ABHPS} for explicit formulas for
$f$-polynomials for all simply laced root polytopes).
The assertions appear for $\vla = \lie{g}$ in \cite{CM} as their ``main
results'' Theorem 1.2(1), Theorem 1.2(2), Theorem 1.1(2), and Theorem 1.3
respectively.
Theorem \ref{Tfinite} now shows that these results hold for a very large
family of highest weight modules $\vla$ (some of them hold for all
$\vla$) and for all semisimple Lie algebras; see Remark \ref{R23}.
Additionally, in this paper we show several other related statements,
which specialize to many of the results in \cite{CM} in the special case
$\lambda = \theta, \vla = L(\theta) = \lie{g}$. See Section
\ref{Scellini} for more details. Note that these results are not known
for any other highest weight module.

\begin{remark}\label{R23}
We briefly remark on the assumption in the last two parts of Theorem
\ref{Tfinite}:
\begin{equation}\label{E23}
\conv_\R \wt \vla = \conv_\R \wt M(\lambda,J(\vla)).
\end{equation}

\noindent We show in Proposition \ref{P23} that equation \eqref{E23}
holds under any of the four hypotheses on $\vla$ that are listed in
Theorem \ref{T23}. Thus Theorem \ref{Tfinite} holds for a very large
class of highest weight modules.
\end{remark}

The third result in this section involves obtaining a minimal half-space
representation for the convex polyhedron $\conv_\R \wt \vla$. It is known
for the adjoint representation (see \cite{CM}) and is not hard to prove
for Verma modules, but it is not known for any nonzero module $\vla \neq
\lie{g}$.

\begin{utheorem}\label{Thalfspace}
Suppose $\lambda \in \lie{h}^*$ and $M(\lambda) \twoheadrightarrow \vla$,
and assume $\conv_\R \wt \vla = \conv_\R \wt M(\lambda, J(\vla))$.
Then the convex polyhedron $\conv_\R \wt \vla$ is represented as the
intersection of the half-spaces
\[ H_{i,w} := \{ \mu \in \lambda - \R \Delta : (w\lambda - \mu,
w\omega_i) \geq 0 \}, \qquad \forall w \in W^i,\ i \in I, \]

\noindent where $W^i$ is any set of coset representatives of $W_{J(\vla)}
/ W_{(I \setminus \{ i \})_{\max}}$. Moreover, the representation is
minimal if one runs over only the simple roots $i \in I_{\min} \sqcup (I
\setminus J(\vla))$ such that $(I \setminus \{ i \})_{\min} = I_{\min}
\setminus \{ i \}$.
\end{utheorem}

We prove additional results in this paper, for instance involving
\textit{longest weights}, which are an analogue of long roots in the
adjoint representation. We also show characterizations of when a weight
is minimal in a standard parabolic subset, or when a standard
parabolic face is a facet.

\begin{remark}\label{Rminmax}
Note that analyzing arbitrary modules $\vla$ and their weights is far
more involved than previous work in the literature on root/Weyl
polytopes. We now discuss some of the technical difficulties that arise
when studying infinite-dimensional highest weight modules $\vla$, but do
not arise for finite-dimensional modules and Weyl polytopes. One such
complication is that the integrability and $W$-invariance properties of
Weyl polytopes do not hold for $\conv_\R \wt \vla$ for
infinite-dimensional modules $\vla$.
Thus it is not readily apparent how to extend results that are known only
for the adjoint representation, or even for all finite-dimensional
modules, to all highest weight modules.

A second, more subtle distinction is that the set $J(\vla)$ equals all of
$I$ for finite-dimensional modules $\vla$ (in fact this is a
characterization of finite-dimensionality). In particular, it is easy to
compute $J(\vla)_{\min}$ for finite-dimensional modules $\vla$: it equals
precisely $I_{\min} = I$ if $\lie{g}$ is simple and $\lambda \neq 0$, as
in previous papers \cite{CM,CDR,CG2} by Cellini, Chari, and their
coauthors. (Clearly $J(\vla)_{\max} = I_{\max} = I$ for all semisimple
$\lie{g}$ and finite-dimensional modules $\vla$.)
In contrast, we work with all highest weight modules over an arbitrary
semisimple Lie algebra. The analysis now is more delicate as one has to
account for the ``non-integrable directions'' $I \setminus J(\vla)$.
\end{remark}

We end by generalizing and completing an analysis initiated by Satake,
Borel--Tits, Vinberg, and Casselman for Weyl polytopes. Recall by Theorem
\ref{T23} that every face of a Weyl polytope, or more generally of
$\conv_\R \wt \vla$ for a very large class of highest weight modules
$\vla$, is of the form $\conv_\R w(\wt_J \vla)$ for $w \in W_{J(\vla)}$
and $J \subset I$. We now completely enumerate the redundancies between
these faces, as well as all inclusions of standard parabolic faces, for
arbitrary highest weight modules $\vla$.

\begin{prop}\label{Pvinberg}
Fix $\lambda \in \lie{h}^*$, $M(\lambda) \twoheadrightarrow \vla$, and $J
\subset I$. Now given $w,w' \in W_{J(\vla)}$ and $J' \subset I$,
\begin{align}\label{Evinberg}
&\ \conv_\R w(\wt_J \vla) = \conv_\R w'(\wt_{J'} \vla) \
\Longleftrightarrow \ w(\wt_J \vla) = w'(\wt_{J'} \vla) \notag\\
\Longleftrightarrow \quad &\ J' \setminus J(\vla) = J \setminus
J(\vla), \quad J_{\min} \subset J' \cap J(\vla) \subset J_{\max}, \quad
w^{-1} w' \in W_{J_{\max}}.
\end{align}

\noindent If $\vla_\lambda = \C v_\lambda$, the following are also
equivalent:
(a) $\wt_J \vla \subset \wt_{J'} \vla$;\break
(b) $\conv_\R \wt_J \vla \subset \conv_\R \wt_{J'} \vla$;
(c) $U(\lie{g}_J) v_\lambda \subset U(\lie{g}_{J'}) v_\lambda$;
(d) $J \setminus J(\vla) \subset J' \setminus J(\vla)$
and $J_{\min} \subset J'_{\min}$.
\end{prop}

In particular, equation \eqref{Evinberg} completely classifies the fibers
of the maps $\scrf_{\vla}$ and $\scrf^{(1)}_{\vla}$, defined in equations
\eqref{Efacemap}, \eqref{Efacemap1} respectively. Special cases of this
result were partially known for finite-dimensional modules $L(\lambda)$,
and specifically, the adjoint representation. See Theorems \ref{Tvin} and
\ref{Tcm1}.

\begin{proof}
The first equivalences in and following \eqref{Evinberg} both follow from
Theorem \ref{T5} and the fact that $\wt \vla \cap \conv_\R w(\wt_J \vla)
= w(\wt_J \vla)$ for all $w \in W_{J(\vla)}$ and $J \subset I$, since
$\wt \vla$ is $W_{J(\vla)}$-stable. Now $w(\wt_J \vla) = w'(\wt_{J'}
\vla)$ if and only if $\wt_J \vla = \wt_{J'} \vla$ and $w^{-1} w'$
stabilizes this set. Equation \eqref{Evinberg} now follows by Theorems
\ref{T5} and \ref{Tfinite}.
The last equivalence in the proposition now follows using Theorem
\ref{Tfinite} and by modifying the proof of $(2) \implies (4)$ in Theorem
\ref{T5}.
\end{proof}
%}}}

\section{Inclusion relations among standard parabolic subsets}

In this section we prove Theorem \ref{T5}, which classifies when two
standard parabolic subsets of $\wt \vla$ are equal. (Note by Theorem
\ref{T23} that this is also equivalent to the problem of studying
inclusion relations among maximizer sets of weights, for a very large
family of highest weight modules.)
As the proof of Theorem \ref{T5} is quite involved, we begin by first
studying the case where the standard parabolic subsets in question are
finite. We then proceed to the general case.

%{{{1 Section 4.1 - Inclusion relations among finite maximizer subsets
\subsection{Inclusion relations among finite maximizer subsets}

Recall by Theorem \ref{T1} that the standard parabolic subsets of $\wt
\vla$ that are finite sets are all of the form $\wt_J \vla$ with $J
\subset J(\vla)$. We now characterize when two finite standard parabolic
subsets of $\wt \vla$ are equal.

\begin{prop}\label{P11}
Fix $\lambda \in \lie{h}^*$, $M(\lambda) \twoheadrightarrow \vla$, and
$J,J' \subset J(\vla)$. Then the vertices of $\conv_\R(\wt_J \vla)$ are
precisely $W_J(\lambda)$. Moreover, the following are equivalent
(notation in Definition \ref{Dsupp}):
\begin{enumerate}
\item $\wt_J \vla = \wt_{J'} \vla$.
\item $\rho_{\wt_J \vla} = \rho_{\wt_{J'} \vla}$.
\item $\rho_{\wt_J \vla} \in \Q_+ \rho_{\wt_{J'} \vla}$.
\item $\pi_{J(\vla)}(\rho_{\wt_J \vla}) = \pi_{J(\vla)}(\rho_{\wt_{J'}
\vla})$.
\item $\pi_{J(\vla)}(\rho_{\wt_J \vla}) \in \Q_+
\pi_{J(\vla)}(\rho_{\wt_{J'} \vla})$.
\item $W_J(\lambda) = W_{J'}(\lambda)$.
\item $\rho_{\wt_J \vla}, \rho_{\wt_{J'} \vla}$ are both fixed by $W_{J
\cup J'}$.
\end{enumerate}
\end{prop}

This is an ``intermediate'' result since $J,J' \subset J(\vla)$. The case
of general $J,J' \subset I$ is Theorem \ref{T5}.

The proof of Proposition \ref{P11} requires the following notation and
results from \cite{KR,Khwf}.

\begin{defn}\label{Dsupp}\hfill
\begin{enumerate}
\item Given $J \subset I$, define $\varpi_J : \lambda + \C \Delta_J \to
\pi_J(\lambda) + \C \Delta_J$ (where the codomain comes from $\lie{g}_J$)
as follows: $\varpi_J(\lambda + \mu) := \pi_J(\lambda) + \mu$.
\item Given a nonempty finite subset $X \subset \lie{h}^*$, define its
average value, or barycenter, to be: $\avg(X) := \frac{1}{|X|} \sum_{x
\in X} x$. Also define $\rho_X := \sum_{x \in X} x$.

\item Given $X \subset \lie{h}^*$ and $\mu \in \lie{h}^*$, define the
corresponding \textit{maximizer subset} $X(\mu) := \{ x \in X : (\mu,
x-x') \in \R_+\ \forall x' \in X \}$.
\end{enumerate}
\end{defn}

We now state various results that are repeatedly used in the present and
subsequent sections, to prove the main theorems in this paper. First
recall that the following special case of Proposition \ref{P11} has been
shown in the literature, for finite-dimensional modules.

\begin{theorem}[Khare and Ridenour, {\cite[Theorem 4]{KR}}]\label{Tkr3}
Suppose $\lambda \in P^+$ and $J,J' \subset I = J(L(\lambda))$. The
vertices of $\conv_\R \wt_J L(\lambda)$ are precisely $W_J(\lambda)$.
Moreover, $\wt_J L(\lambda) = \wt_{J'} L(\lambda)$ if and only if
$\rho_{\wt_J L(\lambda)} = \rho_{\wt_{J'} L(\lambda)}$, if and only if
$W_J(\lambda) = W_{J'}(\lambda)$.
\end{theorem}

The following result discusses how to go from the highest weight down to
any other weight in $\wt \vla$.

\begin{lemma}[{\cite[Lemma 3.12]{Khwf}}]\label{Lweights}
Suppose $M(\lambda) \twoheadrightarrow \vla$ (with highest weight space
$\C v_\lambda$) and $\mu \in \wt_J \vla $, for some $\lambda \in
\lie{h}^*$ and $J \subset I$. Then there exist $\mu_j \in \wt_J \vla$
such that
\[ \lambda = \mu_0 > \mu_1 > \dots > \mu_N = \mu, \qquad \mu_j -
\mu_{j+1} \in \Delta_J\ \forall j, \qquad N \geq 0. \]

\noindent Moreover, if $\vla = L(\lambda)$ is simple, then so is the
$\lie{g}_J$-submodule $\vla_J := U(\lie{g}_J) v_\lambda$.
\end{lemma}

The next result is a ``transfer principle'', sending an arbitrary highest
weight module $\vla$ to its ``integrable top'' $\vla_{J(\vla)}$. In other
words, $\varpi_J : \wt_J \vla \to$\break
$\wt L_J(\pi_J(\lambda))$ is a bijection if $J \subset J(\vla)$.

\begin{lemma}[{\cite[Lemma 4.3]{Khwf}}]\label{Lfacts}
Fix $\lambda \in \lie{h}^*, M(\lambda) \twoheadrightarrow \vla$ generated
by $0 \neq v_\lambda \in \vla_\lambda$, and $J \subset I$.
\begin{enumerate}
\item $J \subset J_\lambda$ if and only if $\pi_J(\lambda) \in P^+$ (in
fact, in $P_J^+$).

\item Let $\vla_J := U(\lie{g}_J) v_\lambda$. Then for all $J,J' \subset
I$, $\wt_{J'} \vla_J = \wt_{J \cap J'} \vla$.

\item $\vla_J$ is a highest weight $\lie{g}_J$-module with highest weight
$\pi_J(\lambda)$. In other words, $M_J(\pi_J(\lambda)) \twoheadrightarrow
U(\lie{g}_J) v_\lambda$, where $M_J$ denotes the corresponding Verma
$\lie{g}_J$-module.

\item For all $w \in W_J$ and $\mu \in \C \Delta_J$,
$w(\varpi_J(\lambda + \mu)) = \varpi_J(w(\lambda + \mu))$.
\end{enumerate}
\end{lemma}

Also recall previous results on the supports of barycenters of finite
standard parabolic faces, as well as on maximizing linear functionals
corresponding to standard parabolic faces.

\begin{prop}[{\cite[Proposition 4.10]{Khwf}}]\label{Pstable}
Fix $\lambda \in \lie{h}^*,\ M(\lambda) \twoheadrightarrow \vla$, and $J
\subset J(\vla)$.
\begin{enumerate}
\item Then $\rho_{\wt_J \vla}$ is $W_J$-invariant, and in $P^+_{J_\lambda
\setminus J} \times \C \Omega_{I \setminus J_\lambda}$.

\item Define $\rho_{I \setminus J} := \sum_{i \notin J} \omega_i$. Then
for all $J' \subset J_\lambda$, one has an inclusion of maximizer
subsets:
\begin{equation}\label{Efinsub}
\wt_J \vla = (\wt \vla)(\rho_{I \setminus J}) = (\wt_{J(\vla)}
\vla)(\pi_{J(\vla)} \rho_{\wt_J \vla}) \subset (\wt \vla)(\pi_{J'}
\rho_{\wt_J \vla})
\end{equation}

\noindent and $0 \leq (\pi_{J'} \rho_{\wt_J \vla})(\wt_J \vla) \in \Z_+$.
\end{enumerate}
\end{prop}

Equipped with the above results, it is now possible to prove Proposition
\ref{P11}.

\begin{proof}[Proof of Proposition \ref{P11}]
The assertion about the vertices follows from Theorem \ref{Tkr3} (for
$\lie{g}_{J(\vla)}$) and Lemma \ref{Lfacts}, via the bijection
$\varpi_{J(\vla)}$. In the course of this reasoning, we use that
$W_{J(\vla)}(\lambda) \subset \lambda - \Z_+ \Delta$, and similarly for
$W_{J(\vla)}(\pi_{J(\vla)}(\lambda))$. Next, $\wt_J \vla$ and $\wt_{J'}
\vla$ are both finite sets by Theorem \ref{T1}. The following
implications are now obvious:
\[ (1) \implies (2) \implies (3) \implies (5); \qquad (2) \implies (4)
\implies (5). \]

\noindent Now if (5) holds, then the two (equal) weights have the same
maximizer by Proposition \ref{Pstable}:
\[ \wt_J \vla = (\wt_{J(\vla)} \vla)(\pi_{J(\vla)} \rho_{\wt_J \vla}) =
(\wt_{J(\vla)} \vla)(\pi_{J(\vla)} \rho_{\wt_{J'} \vla}) = \wt_{J'} \vla.
\]

\noindent This proves (1) again. Now if $\wt_J \vla = \wt_{J'} \vla$,
then their convex hulls (which are polytopes) are equal. Via
$\varpi_{J(\vla)}$,  this also means that the convex hulls of certain
subsets of weights of $M := L_{J(\vla)}(\pi_{J(\vla)}(\lambda))$, a
finite-dimensional $\lie{g}_{J(\vla)}$-module, are equal. Hence the sets
of vertices are the same, so by Theorem \ref{Tkr3},
$W_J(\pi_{J(\vla)}(\lambda)) = W_{J'}(\pi_{J(\vla)}(\lambda))$ in $\wt
M$. But then the same holds in $\wt \vla$ via $\varpi_{J(\vla)}$ (using
Lemma \ref{Lfacts}).

Conversely, assume (6); again use Lemma \ref{Lfacts} and work inside $M
=$\break
$L_{J(\vla)}(\pi_{J(\vla)}(\lambda))$ (via $\varpi_{J(\vla)}$). Theorem
\ref{Tkr3} for $\lie{g}_{J(\vla)}$ shows that $\wt_J M = \wt_{J'} M$, so
$\wt_J \vla = \wt_{J'} \vla$ via $\varpi_{J(\vla)}$. Finally, $(7)
\implies (1)$ using Lemma \ref{Lweyl} (below), and conversely, $X :=
\wt_J \vla = \wt_{J'} \vla$ is stable under both $W_J$ and $W_{J'}$ by
Theorem \ref{T1}. Hence so is $\rho_X$, which shows (7).
\end{proof}

The previous proof and the proof of Theorem \ref{T5} use the following
two preliminary results.

\begin{lemma}\label{Lsimple}
Fix $\lambda \in \lie{h}^*$, $M(\lambda) \twoheadrightarrow \vla$, and
$I_0 \subset I$ such that $\vla_{I_0} := U(\lie{g}_{I_0}) v_\lambda$ is a
simple $\lie{g}_{I_0}$-module.
Then the following are equivalent for $J \subset I_0$:
\begin{enumerate}
\item $\wt_J \vla_{I_0} = \wt_{\emptyset} \vla_{I_0} = \{ \lambda \}$.

\item $\lambda - \alpha_j \notin \wt \vla_{I_0}\ \forall j \in J$.

\item $x_j^- v_\lambda = 0\ \forall j \in J$.

\item $x_j^- v_\lambda \in \ker \lie{n}^+\ \forall j \in J$.

\item $J \subset I \setminus \supp(\lambda)$, i.e., $(\lambda,\alpha_j) =
0\ \forall j \in J$.
\end{enumerate}

\noindent Moreover, if $J \cap \supp(\lambda) \neq J' \cap
\supp(\lambda)$ (for $J,J' \subset I_0$), then $\wt_J \vla \neq \wt_{J'}
\vla$. In particular, the assignment $: J \mapsto \wt_J \vla$ is
one-to-one on the power set of $I_0 \cap \supp(\lambda)$.
\end{lemma}

A special case is $\vla_{I_0} = L(\lambda)$ (for any $\lambda \in
\lie{h}^*$), when $\vla = L(\lambda)$ and $I_0 = I$.

\begin{proof}
That $(1) \implies (2) \implies (3) \implies (4)$ is clear.
Next, given (4), $0 = x^+_{\alpha_j} x^-_{\alpha_j} v_\lambda =
\lambda(h_j) v_\lambda$, whence $\lambda(h_j) = 0$. Thus
$(\lambda,\alpha_j) = 0\ \forall j \in J$, whence $J \subset I \setminus
\supp(\lambda)$.

We now show all the contrapositives.
Suppose $\lambda > \mu = \lambda - \sum_{j \in J} a_j \alpha_j \in \wt_J
\vla_{I_0} = \wt_J \vla$ (by Lemma \ref{Lfacts}). By Lemma
\ref{Lweights}, there exists a sequence $\lambda = \mu_0 > \mu_1 > \dots
> \mu_N = \mu$ in $\wt_J \vla$, such that $\mu_j - \mu_{j+1} \in
\Delta_J\ \forall j$. Thus, $\mu_1 = \lambda - \alpha_j \in \wt \vla$ for
some $j \in J$, which contradicts (2). In turn, this implies:
$x^-_{\alpha_j} v_\lambda \neq 0$ (notation as in Lemma \ref{Lweights}),
which contradicts (3).
If (3) fails, then $x^-_{\alpha_j} v_\lambda$ is not a maximal vector
(i.e., not in $\ker \lie{n}^+$), since $\vla_{I_0}$ is simple.
If (4) is false, then by the Serre relations, $0 \neq x^+_{\alpha_j}
x^-_{\alpha_j} v_\lambda = \lambda(h_j) v_\lambda$. Hence $(\lambda,
\alpha_j) \neq 0$, i.e., $j \in \supp(\lambda)$. This contradicts (5).

Finally, given $J,J' \subset I_0$ as above, choose $j \in J \cap
\supp(\lambda) \setminus J'$. By the above equivalences (in which $J = \{
j \}$), $\lambda - \alpha_j \in \wt \vla_{I_0}$. Hence $\lambda -
\alpha_j \in \wt_J \vla_{I_0} \setminus \wt_{J'} \vla_{I_0}$, whence
$\wt_J \vla_{I_0} \neq \wt_{J'} \vla_{I_0}$. By Lemma \ref{Lfacts},
$\wt_J \vla \neq \wt_{J'} \vla$ (since $J,J' \subset I_0$).
\end{proof}

\begin{lemma}\label{Lweyl}
Suppose either that the setup of Proposition \ref{P11} holds and $W_{J
\cup J'}$ fixes $\rho_{\wt_{J'} \vla}$; or suppose $J' \subset I$, $J
\subset J(\vla)$, and $\Delta_J$ is orthogonal to $\lambda$ and to
$\Delta_{J'}$. Then $\wt_{J'} \vla = \wt_{J \cup J'} \vla$.
\end{lemma}

\begin{proof}
First assume that the setup of Proposition \ref{P11} holds. Suppose the
conclusion fails, i.e.,
\begin{equation}\label{Eweights1}
\mu = \lambda - \sum_{j \in J'} a_j \alpha_j - \sum_{j \in J \setminus
J'} a_j \alpha_j \in \wt_{J \cup J'} \vla \setminus \wt_{J'} \vla.
\end{equation}

\noindent As in the proof of Lemma \ref{Lweights}, produce a monomial
word $0 \neq x_{\alpha_{i_N}}^- \cdots x_{\alpha_{i_1}}^- v_\lambda \in
\vla_\mu$.
Then all indices are in $J \cup J'$; choose the smallest $k$ such that
$i_k \in J \setminus J'$, and define $\mu_{k-1} := \lambda -
\sum_{l=1}^{k-1} \alpha_{i_l} \in \wt_{J'} \vla$.
Now, $(\mu_{k-1}, \alpha_{i_k}) = (\lambda, \alpha_{i_k}) -
\sum_{l=1}^{k-1} (\alpha_{i_l}, \alpha_{i_k})$, and each term in the sum
is nonpositive since $i_l \in J', i_k \in J \setminus J'$. Since
$\alpha_{i_k} \in \Delta_{J \setminus J'} \subset \Delta_{J(\vla)}$,
hence $(\mu_{k-1}, \alpha_{i_k}) \geq 0$.

We first claim that $(\mu_{k-1}, \alpha_{i_k}) > 0$. Suppose not. Then
$(\lambda, \alpha_{i_k}) = (\alpha_{i_l}, \alpha_{i_k}) = 0\ \forall l <
k$, whence by the Serre relations, $[x^-_{\alpha_{i_l}},
x^-_{\alpha_{i_k}}] = 0\ \forall l < k$. Hence by the previous paragraph, 
\begin{equation}\label{Eweights2}
0 \neq x_{\alpha_{i_k}}^- \cdots x_{\alpha_{i_1}}^- v_\lambda
= x_{\alpha_{i_{k-1}} }^- \cdots x_{\alpha_{i_1}}^- x_{\alpha_{i_k}}^-
v_\lambda.
\end{equation}

\noindent In particular, $x^-_{\alpha_{i_k}} v_\lambda \neq 0$. But this
contradicts Lemma \ref{Lsimple} (with $J = I_0 = \{ i_k \} \subset
J(\vla)$), since $(\lambda, \alpha_{i_k}) = 0$. This proves the claim.
Moreover, as shown above for $\mu_{k-1}$, $(\mu, \alpha_{i_k}) \geq 0\
\forall \mu \in \wt_{J'} \vla$. Hence $(\rho_{\wt_{J'} \vla},
\alpha_{i_k}) > 0$ from the above analysis. But this contradicts the
$W_{J \cup J'}$-invariance of $\rho_{\wt_{J'} \vla}$, since $\alpha_{i_k}
\in \Delta_{J \setminus J'} \subset \Delta_{J \cup J'}$. This shows the
first assertion.

The second assertion is shown by essentially repeating the above proof;
here is a quick sketch. Suppose again that $\mu \in \wt \vla$ satisfies
\eqref{Eweights1}. Produce a monomial word $0 \neq x_{\alpha_{i_N}}^-
\cdots x_{\alpha_{i_1}}^- v_\lambda \in \vla_\mu$.
Choose the smallest index $k$ such that $i_k \in J \setminus J'$; then
\eqref{Eweights2} holds as well, since $(\alpha_{i_k}, \alpha_{i_l}) = 0$
for all $0 < l < k$ by assumption. Now since $i_k \in J \subset J(\vla)$
and $(\lambda, \alpha_{i_k}) = 0$, it follows that $x_{\alpha_{i_k}}^-
v_\lambda = 0$ in $\vla$, which contradicts \eqref{Eweights2}. Thus no
weight $\mu$ of the form \eqref{Eweights1} exists.
\end{proof}
%}}}

%{{{1 Section 4.2 - Proof of Theorem A
\subsection{Proof of Theorem \ref{T5}}

Having proved Proposition \ref{P11}, we can show our (first) main result.

\begin{proof}[Proof of Theorem \ref{T5}]
First note that the sets of simple roots used in the formulas in
equations \eqref{Emin} \and \eqref{Emax} indeed depend only on
$\supp(\lambda), J$, and $J \setminus J(\vla) = J_2(\vla)$. Now let
$J_{\min}, J_{\max}$ denote the (first) expressions on the right-hand
sides of equations \eqref{Emin}, \eqref{Emax} respectively. We now prove
the various implications in the result, and also show the minimality and
maximality of these expressions $J_{\min}, J_{\max}$ respectively. The
proof is divided into steps for ease of exposition.\medskip

\noindent $\boldsymbol{(2) \Longleftrightarrow (3)}.$
We first record the following fact, and use it without reference in the
rest of the paper. Given $k>0$ and $J'_r, J''_s \subset I$,
\begin{equation}\label{Ecomp2}
\bigcap_{r=1}^k \wt_{J'_r} \vla \cap \bigcap_s \conv_\R (\wt_{J''_s}
\vla) = \wt_{\cap_r J'_r \cap_s J''_s} \vla.
\end{equation}

\noindent Now clearly $(2) \implies (3)$; the converse follows because
$(\wt \vla) \cap \conv_\R (\wt_J \vla) = \wt_J \vla$. For the same
reason, the assertions in this theorem are also equivalent to the
following statement:\smallskip

\textit{$(6)$ There exist $w,w' \in W_{J(\vla)}$ such that $w(\conv_\R
\wt_J \vla) = w'(\conv_\R \wt_{J'} \vla)$.}\smallskip

\noindent $\boldsymbol{(5) \implies (2)}.$
Note by (5) that
\[
\wt_{J_{\min} \sqcup (J \setminus J(\vla))} \vla \subset \wt_{J'} \vla
\subset \wt_{J_{\max} \sqcup (J \setminus J(\vla))} \vla.
\]

\noindent We now claim that the first and third terms in this chain are
equal. Indeed, note by definition of the sets $J_i(\vla)$ that (with a
slight abuse of notation,)
$J_5(\vla)$ is orthogonal to $\{ \lambda \} \cup J_2(\vla) \cup J_3(\vla)
\cup J_4(\vla)$. It follows from equations \eqref{Emin}, \eqref{Emax}
that $\Delta_{J_{\max} \setminus J_{\min}}$ is contained in
$\Delta_{J(\vla)}$ and orthogonal to $\{ \lambda \} \sqcup
\Delta_{J_{\min} \sqcup J_2(\vla)}$. Applying the second part of Lemma
\ref{Lweyl} then yields the claim, since
\begin{equation}\label{Eformula}
\wt_{J_{\min} \sqcup (J \setminus J(\vla))} \vla = \wt_{J'}
\vla = \wt_{J_{\max} \sqcup (J \setminus J(\vla))} \vla.
\end{equation}

\noindent In particular, the previous equality holds for $J' = J$ (since
(5) does too), and this shows (2).\medskip

\noindent $\boldsymbol{(2) \implies (5)}.$
First note via Theorem \ref{T1} that $J \setminus J(\vla) = \{ i \in I :
\lambda - \Z_+ \alpha_i \subset \wt_J \vla \}$. Thus (2) implies that $J
\setminus J(\vla) = J' \setminus J(\vla)$.
Now suppose $C \subset J_3(\vla) \subset J$ is a connected component of
$J \cap J(\vla)$, such that $\pi_C(\lambda) \neq 0$. Then
\[ \wt_{C \cap J'} \vla = \wt_C \vla \cap \wt_{J'} \vla = \wt_C \vla \cap
\wt_J \vla = \wt_C \vla, \]

\noindent and by \cite[Proposition 5.1]{KR}, the affine hull of $\wt_C
\vla$ is $\lambda - \R \Delta_C$. Hence the same holds for $\wt_{C \cap
J'} \vla$, whence $C \subset J'$. It follows that $J_3(\vla) \subset J'$.

Similarly, suppose $C \subset J_4(\vla)$ is a connected component of $J
\cap J(\vla)$, such that $\pi_C(\lambda) = 0 \neq (\Delta_{J \setminus
J(\vla)}, \Delta_C)$. Assume $(\alpha_{j_2}, \alpha_c) \neq 0$ for some
$j_2 \notin J(\vla)$ and $c \in C$. Since $C \cap (J \setminus J(\vla))$
is empty, it follows that $(\lambda - \alpha_{j_2}, \alpha_c) > 0$. Now
consider the highest weight $\lie{g}_C$-module $M := U(\lie{g}_C)
x_{\alpha_{j_2}}^- v_\lambda$. By integrability, $M$ is
finite-dimensional, hence isomorphic to $L_C(\pi_C(\lambda -
\alpha_{j_2}))$ as $\lie{g}_C$-modules. Since $C$ is connected and
$\pi_C(\lambda - \alpha_{j_2}) \in P_C^+ \setminus \{ 0 \}$ by the above
calculation, once again use \cite[Proposition 5.1]{KR} to obtain that the
affine hull of $\wt M$ is $\lambda - \R \Delta_C$. Hence the same holds
for $\wt_C \vla \cap \wt_{J'} \vla = \wt_{C \cap J'} \vla$, whence $C
\subset J'$. It follows that $J_4(\vla) \subset J'$.
Putting together the above analysis shows that $J_{\min} \subset J'$;
i.e., the expression in equation \eqref{Emin} is indeed minimal as
claimed. It also follows that if $\pi_C(\lambda) \neq 0$ or if $\Delta_C$
is not orthogonal to $\alpha_{j_2}$ for $j_2 \in J_2(\vla)$, then
$\pi_C(\lambda - \alpha_{j_2}) \neq 0$. Therefore $C \subset
J_3(M(\lambda - \alpha_{j_2}, J(\vla)))$, which implies the second
equality in \eqref{Emin}.

Finally, we claim that $J(\vla) \setminus J_{\max}$ is disjoint from
$J'$. The claim would imply that $J' \cap J(\vla) \subset J_{\max}$,
which would complete the proof that $(2) \implies (5)$, and also prove
that the penultimate expression in equation \eqref{Emax} is maximal as
asserted. To show the claim, fix an element
\begin{align*}
j \in J(\vla) &\ \setminus J_{\max} = J_6(\vla) \setminus J_{\max}\\
= &\ (J_6(\vla) \setminus \{ \lambda \}^\perp) \cup
((J_6(\vla) \cap \{ \lambda \}^\perp) \setminus J_2(\vla)^\perp)\\
& \cup ((J_6(\vla) \cap \{ \lambda \}^\perp) \setminus J_{\min}^\perp).
\end{align*}

We show in each of these three cases that $j \notin J'$, which would
complete the proof. First if $j \in J_6(\vla) \setminus \{ \lambda
\}^\perp$, then $s_j(\lambda) \in \wt \vla \setminus \wt_J \vla = \wt
\vla \setminus \wt_{J'} \vla$. Thus $j \notin J'$. Similarly if $j \in
(J_6(\vla) \cap \{ \lambda \}^\perp) \setminus J_2(\vla)^\perp$, then
choose $j' \in J \setminus J(\vla) = J_2(\vla)$ such that $(\alpha_j,
\alpha_{j'}) \neq 0$. Now $\lambda - \alpha_{j'} \in \wt \vla$ and $j \in
J(\vla)$, so Theorem \ref{T1} yields
\[ s_j(\lambda - \alpha_{j'}) = \lambda - s_j(\alpha_{j'}) \in \wt \vla
\setminus \wt_J \vla = \wt \vla \setminus \wt_{J'} \vla. \]

\noindent Once again, it follows that $j \notin J'$. Finally, suppose $j
\in (J_6(\vla) \cap \{ \lambda \}^\perp) \setminus J_{\min}^\perp$.
Choose $j_0 \in J_{\min}$ such that $(\alpha_j, \alpha_{j_0}) \neq 0$.
By equation \eqref{Emin} proved above, there are now two cases:
\begin{itemize}
\item The first possibility is that $j_0 \in J_3(\vla)$, i.e., $j_0$ is
in a connected component $C$ of $J \cap J(\vla)$ such that
$\pi_C(\lambda) \neq 0$. In this case write $\lambda -
w_\circ^C(\lambda) = \sum_{c \in C} n_c \alpha_c$ for $n_c \in
\Z_+$. Then $n_c > 0$ for all $c$ by \cite[Proposition 5.1]{KR}; in
particular, $n_{j_0} > 0$. Now by Theorem \ref{T1},
$s_j(w_\circ^C(\lambda)) \in \wt \vla \setminus \wt_J \vla = \wt \vla
\setminus \wt_{J'} \vla$.

\item Otherwise $j_0 \in J_4(\vla)$, i.e., $j_0$ is in a connected
component $C$ of $J \cap J(\vla)$ for which $\pi_C(\lambda) = 0$ but
$\Delta_C \not\perp \Delta_{J_2(\vla)}$. Suppose $(\Delta_C,
\alpha_{j_2}) \neq 0$ for some $j_2 \in J_2(\vla) = J \setminus J(\vla)$.
Consider the highest weight $\lie{g}_C$-module $M := U(\lie{g}_C)
x_{\alpha_{j_2}}^- v_\lambda$. By integrability, $\dim M < \infty$ and $M
\cong L_C(\pi_C(\lambda - \alpha_{j_2}))$ as $\lie{g}_C$-modules. Since
$C$ is connected and $\pi_C(\lambda - \alpha_{j_2}) \in P_C^+ \setminus
\{ 0 \}$ by choice of $j_2$, it follows via \cite[Proposition 5.1]{KR}
that the affine hull of $\wt M$ is $\lambda - \R \Delta_C$. Now write the
difference of the extremal weights of $M$ as a sum of positive roots;
thus,
$(\lambda - \alpha_{j_2}) - w_\circ^C(\lambda - \alpha_{j_2}) = \sum_{c
\in C} n_c \alpha_c$,
with $n_c > 0$ for all $c \in C$. In particular, $n_{j_0} > 0$.
Recall that $j \in J_6(\vla) \subset J(\vla)$, so $\mu :=
s_j(w_\circ^C(\lambda - \alpha_{j_2}))$ lies in $W_{J(\vla)}(\lambda -
\Delta_{J_2}) \subset \wt \vla$ by Theorem \ref{T1}. On the other hand,
since $(\alpha_j, \alpha_{j_0}) \neq 0$, it follows that $\lambda - \mu =
\sum_{i \in J(\vla)} n_i \alpha_i$ with $n_j > 0$ for $j \in J_6(\vla) =
J(\vla) \setminus J$. Therefore $\mu \in \wt \vla \setminus \wt_J \vla =
\wt \vla \setminus \wt_{J'} \vla$.
\end{itemize}

In either case the above analysis shows that $j \notin J'$. This yields
$J' \cap J(\vla) \subset J_{\max}$, proving that $(2) \implies (5)$. The
equivalence $(2) \Longleftrightarrow (5)$ also shows that the expressions
in equation \eqref{Emin} and the first expression in equation
\eqref{Emax} are indeed the desired, extremal subsets of weights. It is
now easily verified that the last two expressions in equation
\eqref{Emax} are equal, since $J_5(\vla) \subset J(\vla) \cap \{ \lambda
\}^\perp \cap J_{\min}^\perp \cap J_2(\vla)^\perp$.\medskip

\noindent $\boldsymbol{(2) \Longleftrightarrow (4)}.$
Clearly $(4) \implies (2)$. Conversely, we first \textit{claim} that
$U(\lie{g}_J) v_\lambda = U(\lie{g}_{J_{\min} \sqcup (J \setminus
J(\vla))}) v_\lambda$ for all $J \subset I$. Note that the claim,
together with $(2) \Longleftrightarrow (5)$, immediately implies (4).

The claim is proved by showing that each side is contained in the other.
One inclusion is obvious; conversely, $U(\lie{g}_J) v_\lambda$ is spanned
by the set $\mathcal{F}_J$ of words in the alphabet $\{ f_j : j \in J
\}$, applied to $v_\lambda$. Consider a subset $\mathcal{B} \subset
\mathcal{F}_J$ that corresponds to a weight basis of $U(\lie{g}_J)
v_\lambda$. Then using (5),
\[
\wt b \in -\lambda + \wt_J \vla = -\lambda + \wt_{J_{\min} \sqcup (J
\setminus J(\vla))} \vla \subset -\Z_+ \Delta_{J_{\min} \sqcup (J
\setminus J(\vla))}, \quad \forall b \in \mathcal{B}.
\]

\noindent Consequently, $\mathcal{B} \subset \mathcal{F}_{J_{\min} \sqcup
(J \setminus J(\vla))}$. This shows that $U(\lie{g}_J) v_\lambda$ is
contained in $U(\lie{g}_{J_{\min} \sqcup (J \setminus J(\vla))})
v_\lambda$. This proves the claim, and hence that $(2) \implies
(4)$.\medskip

\noindent $\boldsymbol{(1) \Longleftrightarrow (2)}.$
Clearly $(2) \implies (1)$. Conversely, suppose (1) holds and $j \in J'
\setminus J(\vla)$. Then by Theorem \ref{T1} and \cite[Proposition
4.4]{Khwf},
\begin{align*}
w^{-1}w' (\lambda) - \Z_+ (w^{-1} w' \alpha_j) = &\ w^{-1} w'(\lambda -
\Z_+ \alpha_j) \subset w^{-1} w'(\wt_{J'} \vla)\\
= &\ \wt_J \vla \subset \wt_J M(\lambda, J(\vla)).
\end{align*}

\noindent By \cite[Proposition 2.3]{KR}, this implies that $w^{-1}
w'(\alpha_j) \in \Z_+ \Delta_J \cap \Z_+ (\Phi^+ \setminus
\Phi^+_{J(\vla)})$. Therefore,
\[
w^{-1} w'(\alpha_j) \in \Phi \cap (\Z_+ \Delta_J) \cap (\Phi^+ \setminus
\Phi^+_{J(\vla)}) = \Phi^+_J \setminus \Phi^+_{J \cap J(\vla)} \subset
\Phi_{J \cup J(\vla)}.
\]

\noindent This implies that $\alpha_j \in W_{J(\vla)}(\Phi_{J \cup
J(\vla)}) = \Phi_{J \cup J(\vla)}$ for all $j \in J' \setminus J(\vla)$.
In particular, $J' \setminus J(\vla) \subset J \setminus J(\vla)$, and by
symmetry, the reverse inclusion holds as well.

The meat of this implication is in the claim that $J_{\min} \subset J'
\cap J(\vla)$. To show the claim we first describe our approach, in order
to clarify the subsequent detailed exposition. Choose finite,
distinguished $W_{J(\vla)}$-stable subsets $T^\mu \subset \wt \vla$ (for
certain weights $\mu \in Q^+$) and define $T^\mu_J := T^\mu \cap \wt_J
\vla$ for $J \subset I$. Then $w(T^\mu_J) = w'(T^\mu_{J'})$, so that
$w^{-1} w'(\pi_{J(\vla)} (\rho_{T^\mu_{J'}})) =
\pi_{J(\vla)}(\rho_{T^\mu_J})$ (this is only true after transporting the
situation via $\varpi_{J(\vla)}$ to $L_{J(\vla)}(\pi_{J(\vla)}(\lambda -
\mu))$). But $\pi_{J(\vla)}(\rho_{T^\mu_{J'}})$ and
$\pi_{J(\vla)}(\rho_{T^\mu_J})$ are both in $P^+_{J(\vla)}$, so they are
equal, whence their maximizer sets in $T^\mu_J, T^\mu_{J'}$ are equal.
This will yield the result by Proposition \ref{Pstable} by studying
specific sets $T^\mu$ for various $\mu$.

We now explain the details in the preceding paragraph. Define
\begin{equation}\label{Evinberg0}
\begin{aligned}
\mathbb{T}(\vla) := &\ \{ \mu \in Q^+_{I \setminus J(\vla)} : \lambda -
\mu \in \wt \vla \},\\
\forall \mu \in \mathbb{T}(\vla), \ T^\mu := &\ \wt L_{J(\vla)}(\lambda -
\mu) \subset \wt \vla.
\end{aligned}
\end{equation}

\noindent There is a slight abuse of notation here; note that $T^\mu$ is
distinct from\break
$\wt L_{J(\vla)}(\pi_{J(\vla)}(\lambda - \mu))$, which is contained in
$\pi_{J(\vla)}(\lambda) - \R \Delta_{J(\vla)}$. Moreover, $T^\mu$ is a
finite, $W_{J(\vla)}$-stable subset of $\wt \vla$ for all $\mu \in
\mathbb{T}(\vla)$.

Given $J \subset I$, define $T^\mu_J := T^\mu \cap \wt_J \vla$. Now
suppose (1) holds, i.e.,\break
$w^{-1} w'(\wt_{J'} \vla) = \wt_J \vla$.
Intersecting both sides with $T^\mu$ yields:
\[
T^\mu_J = T^\mu \cap w^{-1} w'(\wt_{J'} \vla) = w^{-1} w'(T^\mu_{J'}),
\]

\noindent whence $\rho_{T^\mu_J} = w^{-1} w'(\rho_{T^\mu_{J'}})$. Now
applying Lemma \ref{Lfacts},
\[
w^{-1} w'(\varpi_{J(\vla)}(\rho_{T^\mu_{J'}})) =
\varpi_{J(\vla)}(w^{-1} w'(\rho_{T^\mu_{J'}})) =
\varpi_{J(\vla)}(\rho_{T^\mu_J}).
\]

\noindent Note that $\varpi_{J(\vla)}(\rho_{T^\mu_J}) = \rho_{\wt_{J \cap
J(\vla)} M^\mu}$, where $M^\mu :=
L_{J(\vla)}(\pi_{J(\vla)}(\lambda-\mu))$ is a finite-dimensional
$\lie{g}_{J(\vla)}$-module because $\mu \in I \setminus J(\vla)$.
Similarly for $J'$ in place of $J$. Hence $w^{-1} w'(\rho_{\wt_{J' \cap
J(\vla)} M^\mu}) = \rho_{\wt_{J \cap J(\vla)} M^\mu}$. Apply Proposition
\ref{Pstable} over $\lie{g}_{J(\vla)}$; then $\rho_{\wt_{J \cap J(\vla)}
M^\mu}, \rho_{\wt_{J' \cap J(\vla)} M^\mu} \in P_{J(\vla)}^+$. Since
every $W_{J(\vla)}$-orbit contains at most one dominant element,
$\rho_{\wt_{J \cap J(\vla)} M^\mu} = \rho_{\wt_{J' \cap J(\vla)} M^\mu}$.
Therefore by Proposition \ref{Pstable} over $\lie{g}_{J(\vla)}$, the
corresponding maximizer subsets in $\wt M^\mu$ are equal:
\begin{align*}
\wt_{J \cap J(\vla)} M^\mu = &\ (\wt M^\mu)(\rho_{\wt_{J \cap J(\vla)}
M^\mu})\\
= &\ (\wt M^\mu)(\rho_{\wt_{J' \cap J(\vla)} M^\mu}) = \wt_{J' \cap
J(\vla)} M^\mu.
\end{align*}

Thus the problem is now reduced to a finite-dimensional situation over
the semisimple Lie algebra $\lie{g}_{J(\vla)}$. Introduce the following
notation for convenience:
\begin{equation}\label{Evinberg1}
\widetilde{J} := J \cap J(\vla), \qquad
\widetilde{J'} := J' \cap J(\vla), \qquad
\eta_\mu := \pi_{J(\vla)}(\lambda - \mu).
\end{equation}

\noindent Using this notation, the above analysis in the present step
shows that
\begin{equation}\label{Evinberg2}
w(\wt_J \vla) = w'(\wt_{J'} \vla) \ \implies \ \forall \mu \in
\mathbb{T}(\vla),\ \wt_{\widetilde{J}} L_{J(\vla)}(\eta_\mu) =
\wt_{\widetilde{J'}} L_{J(\vla)}(\eta_\mu).
\end{equation}

\noindent Now apply the equivalence $(2) \Longleftrightarrow (5)$ of this
theorem to equation \eqref{Evinberg2}. This yields:
\begin{equation}\label{Eminmax1}
\widetilde{J}_3(L_{J(\vla)}(\eta_\mu)) = \widetilde{J}_{\min} =
\widetilde{J'}_{\min} = \widetilde{J'}_3(L_{J(\vla)}(\eta_\mu)) \subset
\widetilde{J'} = J' \cap J(\vla).
\end{equation}

\noindent The final step in proving the claim that $J_{\min} \subset J'
\cap J(\vla)$, is to study equation \eqref{Eminmax1} for various special
values of $\mu$, namely, $\mu \in \{ 0 \} \sqcup \Delta_{J_2(\vla)}$.
If $\mu = 0$, then $\widetilde{J}_3(L_{J(\vla)}(\eta_0)) = (J \cap
J(\vla))_{\min} = J_3(\vla)$, so $J_3(\vla) \subset J' \cap J(\vla)$.
Next, if $\mu = \alpha_j$ for $j \in J_2(\vla)$, then by the above
analysis, $\widetilde{J}_3(L_{J(\vla)}(\eta_j))$ is the union of the
connected components $C$ in the Dynkin diagram of $J \cap J(\vla)$, which
satisfy: $\pi_C(\lambda - \alpha_j) \neq 0$. It follows from the
definitions that $J_3(\vla) \sqcup J_4(\vla) \subset J' \cap J(\vla)$. By
equation \eqref{Emin}, it follows that $J_{\min} \subset J' \cap
J(\vla)$, which proves the above claim.

The last step is to note that $J_{\min} \sqcup (J \setminus
J(\vla)) \subset J'$ from the above analysis, so by applying equation
\eqref{Eformula}, $\wt_J \vla = \wt_{J_{\min} \sqcup (J \setminus
J(\vla))} \vla \subset \wt_{J'} \vla$. The reverse inclusion is proved by
symmetry. Therefore $(1) \implies (2)$ holds and the proof is complete.
\end{proof}
%}}}

%{{{1 Section 4.- - Concluding remarks: negative results
\subsection*{Concluding remarks: negative results}

We conclude this section by discussing a couple of related results that
are negative. Given Theorem \ref{T5}, it is natural to ask if the
condition $\wt_J \vla = \wt_{J'} \vla$ is equivalent to the following
``simpler'' conditions:
\begin{equation}\label{Enotminmax}
J \setminus J(\vla) = J' \setminus J(\vla), \qquad \wt_{J \cap J(\vla)}
\vla = \wt_{J' \cap J(\vla)} \vla.
\end{equation}

\noindent The answer to this question is: not always. Indeed $\wt_J \vla
= \wt_{J'} \vla$ implies \eqref{Enotminmax}; however, the converse is not
true because the sets $J_4(\vla)$ and $J'_4(\vla)$ may not coincide. For
a concrete example, let $\lie{g} = \lie{sl}_3$ and consider $\lambda = c
\omega_2$ for $c \in \C$, and $\vla = M(\lambda, \{ 1 \}) = U(\lie{g}) /
U(\lie{g}) (\ker \lambda + \lie{n}^+ + \C x_1^-)$. Thus $I = \{ 1, 2 \}$
and $J(\vla) = \{ 1 \}$. Now it is easily verified that $J = \{ 1, 2 \}$
and $J' = \{ 2 \}$ satisfy equation \eqref{Enotminmax}, since $\wt_{ \{ 1
\}} \vla = \{ \lambda \} = \wt_\emptyset \vla$. On the other hand,
$s_1(\lambda - \alpha_2) = \lambda - \alpha_1 - \alpha_2$, so $\wt_{J'}
\vla = \lambda - \Z_+ \alpha_2 \subsetneq \wt_J \vla$.

A related observation is that an approach based on equation
\eqref{Enotminmax} leads only up to $J_3(\vla)$, while $J_{\min} =
J_3(\vla) \sqcup J_4(\vla)$. However this is not an obstruction if one
recalls that by equation \eqref{Emin}, $J_{\min}$ can be expressed only
using ``$J_3$-type'' sets for various highest weight modules.

A second question arises upon observing that if $\wt_{J'} \vla = \wt_{J
\cup J'} \vla$, then obviously $\wt_J \vla \subset \wt_{J'} \vla$. Given
Theorem \ref{T5}, it is natural to ask if the converse always holds as
well. It turns out that this is not the case; for example, suppose $\lie
g = \lie{sl}_3$, $\lambda = (c+1) \omega_2 \in P^+ \setminus \{ 0 \}$
with $c \in \Z_+$, and $\vla = L(\lambda)$ is simple. Then,
\[ \wt_{\{ 1 \}} L(\lambda) = \{ \lambda \} \subsetneq \wt_{\{ 2 \}}
L(\lambda) \subsetneq \wt_{\{ 1, 2 \}} L(\lambda) = \wt L(\lambda). \]
%}}}

%{{{1 Section 5 - Faces of highest weight modules: combinatorial results 
\section{Faces of highest weight modules: combinatorial results} 

In this section we apply Theorem \ref{T5} in order to study highest
weight modules in greater detail. The goal of this section is to prove
Theorem \ref{Tfinite}. We begin by explaining Remark \ref{R23}, which
discussed how the notion of a Weyl polytope was extended in \cite{Khwf}
to apply to general highest weight modules.

\begin{prop}\label{P23}
Suppose $(\lambda, \vla)$ satisfy any of the four assumptions in Theorem
\ref{T23}:
(a) $\lambda(h_i) \neq 0\ \forall i \in I$ and $\vla$ is arbitrary;
(b) $|J_\lambda \setminus J(\vla)| \leq 1$ (e.g., if $\vla$ is simple for
any $\lambda \in \lie{h}^*$);
(c) $\vla = M(\lambda, J')$ for some $J' \subset J_\lambda$; or
(d) $\vla$ is pure (in the sense of \cite{Fe}).

Then equation \eqref{E23} holds: $\conv_\R \wt \vla = \conv_\R \wt
M(\lambda, J(\vla))$. In turn, equation \eqref{E23} implies all of the
conclusions in Theorem \ref{T23}.
\end{prop}

\begin{proof}
It is not hard to show that both parts of this result follow from the
proofs of \cite[Theorems B and C]{Khwf}. In fact, the condition
\eqref{E23} implies all of the conclusions of \cite[Theorems B and
C]{Khwf}.
\end{proof}

In order to prove Theorem \ref{Tfinite}, two additional preliminary
results are required. The first result involves the barycenter of a
finite standard parabolic subset of weights of $\wt \vla$.

\begin{lemma}\label{Lfinite}
If $\lambda \in \lie{h}^*, M(\lambda) \twoheadrightarrow \vla$, and
$J \subset I$, then $\avg(\wt_{J \cap J(\vla)} \vla) = \avg(W_{J \cap
J(\vla)}(\lambda))$. In other words, the barycenter of the set $\wt_{J
\cap J(\vla)} \vla$ coincides with that of the vertices of its convex
hull; moreover, this vector lies in $\Q_+ \Omega_{J_\lambda \setminus (J
\cap J(\vla))_{\max}} \times \C \Omega_{I \setminus J_\lambda}$.
\end{lemma}

Note that this result specializes to \cite[Theorem 1.2(3)]{CM} when
$\lie{g}$ is simple and $\vla$ is the adjoint representation $\lie{g} =
L(\theta)$ (via the dictionary mentioned in Section \ref{Scellini}). The
result also extends Proposition \ref{Pstable}(1). Further note that $(J
\cap J(\vla))_{\max}$ can be computed using Theorem \ref{T5}.

\begin{proof}
Consider the $\lie{g}_{J(\vla)}$-submodule $U(\lie{g}_{J(\vla)})
v_\lambda \cong L_{J(\vla)}(\lambda)$ of $\vla$.
Use Lemma \ref{Lfacts} and \cite[Proposition 5.2]{KR} to obtain that
\begin{equation}\label{Eaverage}
\avg(\wt_{J \cap J(\vla)} \vla) - (\lambda - \pi_{J(\vla)}(\lambda)) =
\avg (W_{J \cap J(\vla)}(\lambda)) - (\lambda - \pi_{J(\vla)}(\lambda)).
\end{equation}
This proves the first equality. The last assertion follows from
Proposition \ref{Pstable} and Theorem \ref{T5}, with $J$ replaced by $J
\cap J(\vla)$.
\end{proof}

The second result proves some of the assertions in Theorem \ref{Tfinite},
including a stabilizer subgroup computation in the finite-dimensional
setting.

\begin{prop}\label{Pfinite}
Suppose $\lambda \in \lie{h}^*$, $M(\lambda) \twoheadrightarrow \vla$,
and $J \subset I$.
\begin{enumerate}
\item In the Weyl group $W$, $W_{J_{\max}} = W_{J_{\min}} \times
W_{J_{\max} \setminus J_{\min}}$. Here $W_{J_{\max} \setminus J_{\min}}$
fixes the face $\conv_\R \wt_J \vla$ pointwise, while no element of
$W_{J_{\max}} \setminus W_{J_{\max} \setminus J_{\min}}$ does so.

\item Let $J \subset J(\vla)$. Then the stabilizer subgroups in
$W_{J(\vla)}$ of $\conv_\R (\wt_J \vla)$ and of (the average of) $\wt_J
\vla$ agree, and equal $W_{J_{\max}}$.
\end{enumerate}
\end{prop}

\begin{proof}
(1) It follows from equations \eqref{Emin} and \eqref{Emax} and the
definitions that\break
$\Delta_{J_{\max} \setminus J_{\min}}$ is orthogonal to $\{ \lambda \}
\sqcup \Delta_{J_{\min} \sqcup J_2(\vla)}$. Therefore by Theorem
\ref{T5}, $W_{J_{\max} \setminus J_{\min}}$ fixes $\wt_J \vla$ (and hence
its convex hull) pointwise. It also follows that $W_{J_{\max} \setminus
J_{\min}}$ commutes with $W_{J_{\min}}$ in $W$.

It remains to prove no element of $W_{J_{\max}} \setminus
W_{J_{\max} \setminus J_{\min}}$ fixes all of $\conv_\R \wt_J \vla$.
Indeed, suppose $w \in W_{J_{\max}}$ fixes $\conv_\R \wt_J \vla$
pointwise. Write $w = w_1 w_2$, where $w_1 \in W_{J_{\min}}, w_2 \in
W_{J_{\max} \setminus J_{\min}}$. Then $w$ fixes $\conv_\R \wt_J \vla$
pointwise if and only if $w_1$ does so. We claim this happens if and only
if $w_1 = 1$. In fact, we show the stronger statement that no nontrivial
$w \in W_{J_{\min}}$ fixes $\wt_J \vla$.

To show this statement, first note that any $w \in W_{J_{\min}}$ which
fixes $\wt_J \vla$ must fix $\lambda$ and $\lambda - \alpha_{j_2}$ for
all $j_2 \in J_2(\vla)$, so it fixes $\Delta_{J \setminus J(\vla)}$.
Next, the minimality of $(J \cap J(\vla))_{\min} = J_3(\vla)$ implies
shows that for all $j_3 \in J_3(\vla)$, there exists a weight $\mu \in
\wt_{J_3(\vla)} \vla \subset \wt_J \vla$ such that $\lambda - \mu =
\sum_{j \in J_3(\vla)} c_j \alpha_j$ with $c_{j_3} > 0$. Using Lemma
\ref{Lweights}, it follows that $w$ fixes $\alpha_{j_3}$ for all $j_3 \in
J_3(\vla)$.
Next, if $j_4 \in J_4(\vla)$ then let $C$ be the connected component
of the Dynkin diagram of $J \cap J(\vla)$ such that $j_4 \in C \subset
J_4(\vla)$. Choose $j_2 \in J_2(\vla)$ such that $(\alpha_{j_2},
\alpha_{j_4}) \neq 0$; then $\pi_C(\lambda - \alpha_{j_2}) \neq 0$. Now
recall from the proof of $(1) \implies (2)$ in Theorem \ref{T5} that $\wt
L_{J(\vla)}(\lambda - \alpha_{j_2}) \subset \wt_J \vla$. Since $w$ fixes
$\lambda$ as well as $\Delta_{J_2(\vla) \sqcup J_3(\vla)}$,
an argument similar to that for $J_3(\vla)$ above shows that $w$ also
fixes $\alpha_{j_4}$, and hence all of $\Delta_{J_4(\vla)}$. It follows
by equation \eqref{Emin} that $w \in W_{J_{\min}}$ fixes
$\Delta_{J_{\min}}$, and hence sends no positive root in
$\Phi_{J_{\min}}$ to $\Phi^-$. Therefore $w$ has length zero in
$W_{J_{\min}}$, i.e.~$w=1$ as claimed.\medskip

(2) Define the sets $S_j \subset \lie{h}^*$ for $1 \leq j \leq 5$ as
follows:
\begin{alignat*}{7}
S_1 := &\ \wt_J \vla, \qquad &
S_2 := &\ \conv_\R \wt_{J_{\max}} \vla, \qquad &
S_3 := &\ W_{J_{\max}}(\lambda),\\
S_4 := &\ \{ \avg(S_1) \}, \qquad &
S_5 := &\ \{ \pi_{J(\vla)} (\avg(S_1)) \}. &
\end{alignat*}

\noindent Now define $W_j := {\rm stab}_{W_{J(\vla)}} S_j$ to be the
respective stabilizer subgroup in $W_{J(\vla)}$ of $S_j$, for $1 \leq j
\leq 5$. We then claim that $W_{J_{\max}} \subset W_1 \subset \dots
\subset W_5 \subset W_{J_{\max}}$, which shows that these subgroups are
all equal and proves this part.

To prove the claim, note by Theorem \ref{T5} that $S_1 = \wt_{J_{\max}}
\vla$ is stable under $W_{J_{\max}}$, so that $W_{J_{\max}} \subset W_1$.
Next, it is clear that $W_1 \subset W_2$, and Proposition \ref{P11} (with
$J$ replaced by $J_{\max}$) shows that $W_2 \subset W_3$. Moreover,
equation \eqref{Eaverage} shows that $\avg \wt_J \vla = \avg
\wt_{J_{\max}} \vla = \avg S_3$. Hence $W_3 \subset W_4$.
Next, write $\avg(S_1) = \pi_{J(\vla)}(\avg(S_1)) + \pi_{I \setminus
J(\vla)}(\avg(S_1))$. Note that $W_4 \subset W_{J(\vla)}$ fixes the first
and third vectors in this equation. Hence $W_4 \subset W_5$.

It remains to show $W_5 \subset W_{J_{\max}}$. Note by Lemma
\ref{Lfacts}, one can reduce the problem to the case where $I, W$, and
$\lie{g}$ are equal to $J(\vla), W_{J(\vla)}$, and $\lie{g}_{J(\vla)}$
respectively. Now Lemma \ref{Lfinite} implies that $|S_1| \avg(S_1) \in
\Z_+ \Omega_{I \setminus J_{\max}}$ lies in the dominant Weyl chamber. It
follows by assertion (I) in \cite[Chapter V.3.3]{Bou} and Theorem $2$ in
\cite[Chapter VI.1.5]{Bou} that $W_5$ is generated by the simple
reflections it contains. Denote the indices corresponding to these simple
reflections by $J_0$; thus, $J_0 := \{ i \in I : (\avg(S_1),\alpha_i) = 0
\}$. Now note by \cite[Proposition 5.2]{KR} as well as Proposition
\ref{Pstable} that $(\avg(S_1),-)$ is maximized precisely at
$\wt_{J_{\max}} \vla$. Since $J_{\max}$ is maximal in the sense of
Theorem \ref{T5}, it follows that $J_0 \subset J_{\max}$, whence $W_5 =
W_{J_0} \subset W_{J_{\max}}$.
\end{proof}

We now use Theorem \ref{T5} as well as the above analysis in the present
section, to show another of the main results in this paper.

\begin{proof}[Proof of Theorem \ref{Tfinite}]
(1) Note that the affine hull of $\conv_\R \wt_J \vla$ equals the
$\lambda$-translate of the real span of the set $S_{J,\lambda} := \lambda
- \wt_J \vla = \lambda - \wt_{J_{\min} \cup (J \setminus J(\vla))} \vla$
by Theorem \ref{T5}. This implies that $\spa_\R (S_{J,\lambda}) \subset
\R \Delta_{J_{\min} \sqcup (J \setminus J(\vla))}$. Now note that $\Z_+
\alpha_{j_2} \in S_{J,\lambda}$ for all $j_2 \in J \setminus J(\vla)$.
Moreover, the minimality of $(J \cap J(\vla))_{\min} = J_3(\vla)$ implies
that for all $j_3 \in J_3(\vla)$, there exists a weight $\mu_{j_3} \in
S_{J,\lambda}$ such that $\mu_{j_3} = \sum_{j \in J_{\min}} c_j \alpha_j$
with $c_{j_3} > 0$. Using Lemma \ref{Lweights}, it follows that
$\alpha_{j_3} \in \spa_\R (S_{J,\lambda})$ for each $j_3 \in J_{\min}$.

Finally, suppose $j_4 \in C \subset J_4(\vla)$, where $C$ is a connected
component of the Dynkin diagram of $J \cap J(\vla)$. Then there exists
$j_2 \in J_2(\vla)$ such that $(\alpha_{j_2}, \alpha_{j_4}) \neq 0$. Now
recall from the proof of $(1) \implies (2)$ in Theorem \ref{T5} that $\wt
L_{J(\vla)}(\lambda - \alpha_{j_2}) \subset \wt \vla$ by the
$W_{J(\vla)}$-integrability of $\vla$. Since $\pi_C(\lambda -
\alpha_{j_2}) \neq 0$, it follows similar to the above reasoning for
$J_3(\vla)$ that there exists $\mu_{j_4} \in S_{J,\lambda}$ of the form
$\mu_{j_4} = \sum_{j \in J_{\min}} c_j \alpha_j$ with $c_{j_4} > 0$.
Therefore $\Delta_{J_4(\vla)} \subset S_{J,\lambda}$, and hence, $\spa_\R
(S_{J,\lambda}) = \R \Delta_{J_{\min} \sqcup (J \setminus J(\vla))}$.
Taking dimensions of both sides completes the proof of this part.\medskip

(2) This part is the meat of the proof. We first claim that the
stabilizers in $W_{J(\vla)}$ of $\wt_J \vla$ and $\conv_\R \wt_J \vla$
agree. Clearly if $w \in W_{J(\vla)}$ stabilizes $\wt_J \vla$ then it
stabilizes its convex hull. Conversely, if $w \in W_{J(\vla)}$ stabilizes
$\conv_\R \wt_J \vla$, then it stabilizes $(\wt \vla) \cap \conv_\R \wt_J
\vla = \wt_J \vla$, which shows that the two stabilizer subgroups in
$W_{J(\vla)}$ are equal.

Denote this common stabilizer subgroup in $W_{J(\vla)}$ by $W'$. We now
claim that $W' = W_{J_{\max}}$. One inclusion is clear: $W_{J_{\max}
\setminus J_{\min}}$ fixes $\wt_J \vla$; moreover, $W_{J_{\min}}$
preserves $\wt_J \vla = \wt \vla \cap \lambda - \Z_+ \Delta_J$, since
$J_{\min} \subset J$. Therefore $W_{J_{\max}}$ preserves $\wt_J \vla$. To
show the converse inclusion, suppose $w \in W'$ preserves $\wt_J \vla$.
Then $w$ preserves each $W_{J(\vla)}$-stable subset of $\wt_J \vla$. Now
recall the notation in equation \eqref{Evinberg0}; thus $w$ preserves the
sets $T^\mu \cap \wt_J \vla$ for $\mu \in \mathbb{T}(\vla)$. Therefore
one can use Lemma \ref{Lfacts} to transfer the problem to $\wt_{J(\vla)}
L_{J(\vla)}(\pi_{J(\vla)}(\lambda - \mu))$. Then $w$ preserves the set
\[ \varpi_{J(\vla)}(T^\mu_J) = \wt_{J \cap J(\vla)}
L_{J(\vla)}(\pi_{J(\vla)}(\lambda - \mu)) = \wt_{\widetilde{J}_{\max}}
L_{J(\vla)}(\eta_\mu), \]

\noindent where $\widetilde{J}$ and $\eta_\mu$ were defined in equation
\eqref{Evinberg1}. Now denote the stabilizer in $W_{J(\vla)}$ of
$\varpi_{J(\vla)}(T^\mu \cap \wt_J \vla)$ by $W''_\mu$; then
$W_{J_{\max}} \subset W' \subset W''_\mu$.
Thus it suffices to show that
\begin{equation}\label{Evinberg4}
\bigcap_{\mu \in \mathbb{T}(\vla)} W''_\mu \subset W_{J_{\max}}.
\end{equation}

\noindent Note by Proposition \ref{Pfinite} that $W''_\mu =
W_{\widetilde{J}^\mu_{\max}}$, where $\widetilde{J}^\mu_{\max}$ is the
unique maximal set $J'$ of simple roots (by Theorem \ref{T5}) such that
$\wt_{J \cap J(\vla)} L_{J(\vla)}(\eta_\mu) = \wt_{J'}
L_{J(\vla)}(\eta_\mu)$. Moreover, we claim -- akin to the proof of
Theorem \ref{T5} -- that the inclusion in \eqref{Evinberg4} holds even if
the intersection is taken over the smaller set $\mu \in \{ 0 \} \sqcup
\Delta_{J_2(\vla)}$. In other words, the proof is complete if the
following inclusion is shown to hold:
\begin{equation}\label{Evinberg3}
\widetilde{J}^0_{\max} \cap \bigcap_{j \in J_2(\vla)}
\widetilde{J}^{\alpha_j}_{\max} \subset J_{\max}.
\end{equation}

\noindent To prove this inclusion, first note that
$\widetilde{J}^0_{\max} = (J \cap J(\vla))_{\max} \supset J_{\max}$ (the
inclusion follows from equation \eqref{Emax}). We now claim that for all
$j_6 \in (J \cap J(\vla))_{\max} \setminus J_{\max}$, there exists $j \in
J_2(\vla)$ that $j_6 \notin \widetilde{J}^{\alpha_j}_{\max}$. To show
this claim, first compute using equation \eqref{Emax}:
\[
(J \cap J(\vla))_{\max} \setminus J_{\max} = J_6(\vla) \cap \{ \lambda
\}^\perp \cap J_3(\vla)^\perp \setminus (J_2(\vla)^\perp \cap
J_4(\vla)^\perp).
\]

\noindent Thus it remains to consider two cases. The first is if
$j_6 \in J_6(\vla) \cap \{ \lambda \}^\perp \cap J_3(\vla)^\perp
\setminus J_2(\vla)^\perp$. Suppose $(\alpha_{j_6}, \alpha_{j_2}) \neq 0$
for some $j_2 \in J_2(\vla)$. Then
\[
s_{j_6}(\eta_{\alpha_{j_2}}) = s_{j_6}(\pi_{J(\vla)}(\lambda -
\alpha_{j_2})) \in \wt_{\{ j_6 \}} L_{J(\vla)}(\eta_{\alpha_{j_2}})
\setminus \wt_{J \cap J(\vla)} L_{J(\vla)}(\eta_{\alpha_{j_2}}).
\]

\noindent Therefore $j_6 \notin \widetilde{J}_{\max}^{\alpha_{j_2}}$ as
desired.

The second case is when $j_6 \in J_6(\vla) \cap \{ \lambda \}^\perp \cap
J_3(\vla)^\perp \cap J_2(\vla)^\perp \setminus J_4(\vla)^\perp$. Then
there exist $j_2 \in J_2(\vla)$ and a connected component $C \subset
J_4(\vla)$ of $J \cap J(\vla)$, such that neither $\alpha_{j_6}$ nor
$\alpha_{j_2}$ is orthogonal to all of $\Delta_C$. Say $(\alpha_{j_6},
\alpha_{j_4}) \neq 0$ for $j_4 \in C$. We assert that the affine hull of
the set $S' := \wt_{J \cap J(\vla)} L_{J(\vla)}(\eta_{\alpha_{j_2}})$
contains $\lambda - \R \Delta_C$; but this holds by definition of
$J_3(\vla)$ and \cite[Proposition 5.1]{KR}, since
$\pi_C(\eta_{\alpha_{j_2}}) = \pi_C(\lambda - \alpha_{j_2}) \neq 0$. In
particular, the difference $\eta_{\alpha_{j_2}} - w_\circ^{J \cap
J(\vla)}(\eta_{\alpha_{j_2}})$ of the extremal elements in $S'$ lies in
$Q^+ \setminus Q^+_{I \setminus \{ j_4 \}}$. Therefore,
\[
s_{j_6}(w_\circ^{J \cap J(\vla)}(\eta_{\alpha_{j_2}})) \in \wt_{(J \cap
J(\vla)) \sqcup \{ j_6 \}} L_{J(\vla)}(\eta_{\alpha_{j_2}}) \setminus
\wt_{J \cap J(\vla)} L_{J(\vla)}(\eta_{\alpha_{j_2}}),
\]

\noindent which again shows that $j_6 \notin
\widetilde{J}_{\max}^{\alpha_{j_2}}$, proving the claim made after
equation \eqref{Evinberg3}. Putting together the above analysis shows
that \eqref{Evinberg3} holds, whence $W_{J_{\max}} \subset W' \subset
\bigcap_{\mu \in \mathbb{T}(\vla)} W''_\mu \subset W_{J_{\max}}$ as
desired. Finally, that $W_{J_{\max}} = W_{J_{\min}} \times W_{J_{\max}
\setminus J_{\min}}$ was proved in Proposition \ref{Pfinite}.\medskip

Next, suppose $\conv_\R \wt \vla = \conv_\R \wt M(\lambda, J(\vla))$.
Then $\conv_\R \wt_J \vla$ is a face of the convex polyhedron $\conv_\R
\wt \vla$ for all $J \subset I$. We now compute (the size of) the vertex
set of this face. Using Theorems \ref{T23} and \ref{T5}, it follows that
\[
\conv_\R \wt_J \vla = \conv_\R \wt_J M(\lambda) \cap \conv_\R \wt
M(\lambda, J(\vla)) = \conv_\R \wt_J M(\lambda, J(\vla)).
\]

\noindent Thus, applying Theorem \ref{T23} for the $\lie{g}_J$-submodule
$\vla_J := U(\lie{g}_J) v_\lambda \subset M(\lambda, J(\vla))$ shows that
the face $\conv_\R \wt_J \vla$ has vertex set $W_{J \cap
J(\vla)}(\lambda)$. Now use Lemma \ref{Lfacts} to reduce the problem to
studying the vertex set inside the convex hull of weights of the
finite-dimensional $\lie{g}_{J(\vla)}$-module
$L_{J(\vla)}(\pi_{J(\vla)}(\lambda))$.
By Proposition \ref{P11}, the vertex set has size $|W_{J \cap
J(\vla)}(\lambda)| = [W_{J \cap J(\vla)} : W'']$, where $W''$ is the
stabilizer subgroup in $W_{J \cap J(\vla)}$ of $\pi_{J(\vla)}(\lambda)$.
Since $\pi_{J(\vla)}(\lambda) \in P^+$, it follows by assertion (I) in
\cite[Chapter V.3.3]{Bou} and Theorem $2$ in \cite[Chapter VI.1.5]{Bou}
that $W''$ is generated by the simple reflections $s_j$ in it. Now $s_j$
fixes $\pi_{J(\vla)}(\lambda)$ for $j \in J \cap J(\vla)$, if and only if
$j \in J \cap J(\vla) \cap \{ \pi_{J(\vla)}(\lambda) \}^\perp$.
It follows that $W'' = W_{J \cap J(\vla) \cap \{ \pi_{J(\vla)}(\lambda)
\}^\perp}$, and the proof is complete upon noting that $J(\vla) \cap \{
\pi_{J(\vla)}(\lambda) \}^\perp = J(\vla) \cap \{ \lambda \}^\perp$.

Finally, to compute the $f$-polynomial of the convex polyhedron $\conv_\R
\wt \vla$, apply Theorem \ref{T23} for $\vla = M(\lambda, J(\vla))$ to
obtain that every face of the convex polyhedron $\conv_\R \wt \vla$ is
$W_{J(\vla)}$-conjugate to a unique face of the form $\conv_\R
\wt_{J_{\max}} \vla$ (or $\conv_\R \wt_{J_{\min}} \vla$). The result now
follows from the first two parts of this theorem.
\end{proof}

\begin{remark}
Note by the analysis in the penultimate paragraph of the proof of Theorem
\ref{Tfinite} that $[W_J : W_{J \cap \{ \pi_{J(\vla)}(\lambda)
\}^\perp}]$ does not depend on the choice of the set $J$ from among
$[J_{\min}, J_{\max}]$. This can also be seen directly as follows: recall
that $W_{J_{\max} \setminus J_{\min}}$ fixes $\lambda$ and commutes with
$W_{J_{\min}}$, so that $J \cap J(\vla) \setminus J_{\min} \subset \{
\lambda \}^\perp$. Therefore,
\begin{align*}
[W_{J \cap J(\vla)} : &\ W_{J \cap J(\vla) \cap \{ \lambda \}^\perp}]\\
= &\ [W_{J_{\min} \sqcup (J \cap J(\vla) \setminus J_{\min})} :
W_{(J_{\min} \cap \{ \lambda \}^\perp) \sqcup (J \cap J(\vla) \setminus
J_{\min})}]\\
= &\ [W_{J_{\min}} \times W_{J \cap J(\vla) \setminus J_{\min}} :
W_{J_{\min} \cap \{ \lambda \}^\perp} \times W_{J \cap J(\vla) \setminus
J_{\min}}]\\
= &\ [W_{J_{\min}} : W_{J_{\min} \cap \{ \lambda \}^\perp}],
\end{align*}

\noindent and this is indeed independent of $J \in [J_{\min}, J_{\max}]$.
We also remark that Theorem \ref{Tfinite} yields multiple formulas for
the number of vertices of $\conv_\R \wt \vla$, and these formulas are
easily seen to agree in light of the preceding computation.
\end{remark}

We end this section by discussing how a result of Satake, Borel--Tits,
Vinberg, and Casselman for Weyl polytopes with generic dominant integral
highest weight $\lambda \in P^+$ follows from Theorems \ref{T5} and
\ref{Tfinite}. These authors showed that the distinct sets among $\{
\wt_J L(\lambda) : J \subset I \}$ are in bijection with
``$\lambda$-admissible'' subsets $J \subset I$, when $\lambda \in P^+$ is
``admissible''. In the notation of the present paper, $\vla =
M(\lambda,I) = L(\lambda)$ and $J(L(\lambda)) = I$. Moreover, it is not
hard to verify that $\lambda$ being admissible simply means that
$I_{\min} = I$, while $J \subset I$ being $\lambda$-admissible means that
$J = J_{\min} = J_3(L(\lambda))$. We now write down the aforementioned
result, as stated by Vinberg.

\begin{theorem}[{Vinberg, \cite[Proposition 3.2]{Vi}}]\label{Tvinberg}
Fix $\lambda \in P^+$ and $\vla = L(\lambda)$. Suppose $I_{\min} = I$.
Then every face of $\calp(\lambda) = \conv_\R \wt L(\lambda)$ is
$W$-conjugate to $\conv_\R \wt_J L(\lambda)$ for a unique subset $J
\subset I$ such that $J = J_3(L(\lambda)) = J_{\min}$. Moreover, this
face has dimension $|J_{\min}|$.
\end{theorem}

Note that Theorem \ref{Tvinberg} follows from our main Theorems \ref{T5}
and \ref{Tfinite}. In fact these two Theorems show that the assumption
$I_{\min} = I$ is not required to prove Theorem \ref{Tvinberg}. (This
assumption was also used in \cite{Vi} to ensure that the Weyl polytope is
of ``full dimension'' $|I_{\min}| = |I|$.) As discussed in Remark
\ref{Rminmax}, the analysis in Theorems \ref{T5} and \ref{Tfinite} for
general modules $\vla$ is more involved because one has to account for
the simple roots in $J_2(\vla)$ and hence in $J_4(\vla)$.
%}}}

%{{{1 Section 6 - Half-space representation and facets
\section{Half-space representation and facets}

We now study standard parabolic subsets of weights in $\wt \vla$, whose
convex or affine hull has codimension one in $\conv_R \wt \vla$. The goal
in this section is to prove Theorem \ref{Thalfspace}, using the fact that
every convex polyhedron is the intersection of a minimal family of
codimension one facets. In order to do so, it is natural to seek
characterizations of when a particular face $\conv_\R \wt_{I \setminus \{
i \}} \vla$ is a (codimension-one) facet of $\conv_\R \wt \vla$. The
following result provides several different such characterizations, for
all highest weight modules $\vla$.

\begin{prop}\label{Pfacet}
Suppose $\lambda \in \lie{h}^*$ and $M(\lambda) \twoheadrightarrow \vla$.
Given $i \in I$, define the ``coordinate face'' $F_i(\vla) := \conv_\R
\wt_{I \setminus \{ i \}} \vla$. Then the following are equivalent for $i
\in I$:
\begin{enumerate}
\item $F_i(\vla)$ has codimension one in the real affine space $\lambda -
\R \Delta_{I_{\min} \sqcup (I \setminus J(\vla))}$.

\item $i \notin J(\vla) \setminus I_{\min} = I_{\max} \setminus
I_{\min}$, and $F_i(\vla)$ is maximal among the coordinate faces $\{
F_{i'}(\vla) : i' \in I_{\min} \sqcup (I \setminus J(\vla)) \}$.

\item $i \notin J(\vla) \setminus I_{\min}$, and for all $i' \in I_{\min}
\setminus \{ i \}$, there exists $\mu \in \{ 0 \} \sqcup \Delta_{I
\setminus J(\vla)}$ such that the minimal weight $\mu_i$ of
$\wt_{I_{\min} \setminus \{ i \}} L_{J(\vla)}(\pi_{J(\vla)}(\lambda -
\mu))$ satisfies: $(\mu_i, \omega_{i'}) \neq (\lambda, \omega_{i'})$.

\item $(I \setminus \{ i \})_{\min} = I_{\min} \setminus \{ i \}$
and $(I \setminus \{ i \})_{\max} = I_{\max} \setminus \{ i \}$.

\item $i \notin J(\vla) \setminus I_{\min}$, and for all $i' \in I_{\min}
\setminus \{ i \}$, there exists $\mu \in \wt_{I \setminus \{ i \}} \vla$
such that $(\mu, \omega_{i'}) \neq (\lambda, \omega_{i'})$.

\item $i \notin J(\vla) \setminus I_{\min}$, and for all $i' \in I_{\min}
\setminus \{ i \}$, the set $\wt_{I \setminus \{ i \}} \vla$ contains a
nontrivial $\alpha_{i'}$-string.
\end{enumerate}
\end{prop}

We observe (via the dictionary in Section \ref{Scellini}) that
\cite[Theorem 4.5]{CM} is a special case of Proposition \ref{Pfacet},
where $\lie{g}$ is simple and $\vla$ is the adjoint representation. More
precisely, each part of Proposition \ref{Pfacet} corresponds to the same
numbered part in \textit{loc.~cit.}, with the exception of the second
part above, which corresponds to part (7) of \textit{loc.~cit.} (For part
(3), recall the second equality in equation \eqref{Emin}.) We will
address the missing part \cite[Theorem 4.5(2)]{CM} in Corollary
\ref{Ccellini} below.

Note by Theorem \ref{T5} that $I_{\max} = J(\vla)$ by maximality.
Therefore Proposition \ref{Pfacet}(4) reads:
\begin{equation}\label{Eminmax2}
(I \setminus \{ i \})_{\min} = I_{\min} \setminus \{ i \}, \qquad
J(\vla) \setminus \{ i \} = (I \setminus \{ i \})_{\max}.
\end{equation}

\noindent The second equation in \eqref{Eminmax2} is \textit{a priori}
different from what appears in four of the six assertions in Proposition
\ref{Pfacet} -- namely, the condition $i \notin J(\vla) \setminus
I_{\min}$. We now show a preliminary result which is required to prove
Proposition \ref{Pfacet}, and which shows that the aforementioned two
conditions on $i$ are equivalent for all highest weight modules $\vla$.

\begin{prop}\label{Phalfspace}
Suppose $\lambda \in \lie{h}^*$ and $M(\lambda) \twoheadrightarrow \vla$.
Then for all $i \in I$,
\[ (I \setminus \{ i \})_{\max} = \begin{cases}
J(\vla) \setminus \{ i \} & \qquad \text{if } i \in I_{\min} \sqcup (I
\setminus J(\vla)),\\
J(\vla) & \qquad \text{otherwise, i.e., if } i \in J(\vla) \setminus
I_{\min} = I_{\max} \setminus I_{\min}.
\end{cases} \]

\noindent Moreover, the ``coordinate faces'' $\{ \wt_{I \setminus \{ i
\}} \vla : i \in I_{\min} \sqcup (I \setminus J(\vla)) \}$
are distinct and proper subsets of $\wt \vla$, which equals $\wt_{I
\setminus \{ i \}} \vla$ for all $i \in J(\vla) \setminus I_{\min}$.
\end{prop}

Note that the second assertion extends \cite[Proposition 4.1]{CM} from
the adjoint representation (for simple $\lie{g}$) to all highest weight
modules.

\begin{proof}
If $i \in I \setminus J(\vla)$, then $(I \setminus \{ i \})_{\max} =
J(\vla) = J(\vla) \setminus \{ i \}$ by Theorem \ref{T5}. If $i \in
I_{\min}$, then it follows from Theorem \ref{T5} that $\wt_I \vla = \wt
\vla \neq \wt_{I \setminus \{ i \}} \vla$, by taking the intersection
with $\wt_{J(\vla)} \vla$. Therefore $(I \setminus \{ i \})_{\max} \neq
J(\vla)$ as desired. Now suppose $i \in J(\vla) \setminus I_{\min} =
I_{\max} \setminus I_{\min}$. Then $\wt_I \vla = \wt_{I_{\min} \sqcup (I
\setminus J(\vla))} \vla$, so the intermediate set $\wt_{I \setminus \{ i
\}} \vla$ also equals $\wt \vla$. It follows that $(I \setminus \{ i
\})_{\max} = J(\vla)$, proving the formula.

Next, suppose $\wt_{I \setminus \{ i_1 \}} \vla = \wt_{I \setminus \{ i_2
\}} \vla = S'$, say, where $i_1 \neq i_2$. By Theorem \ref{T5}, $i_1, i_2
\in J(\vla)$ and there is a maximal set $J_{\max} \subset I$ such that
$S' = \wt_{J_{\max} \sqcup (I \setminus J(\vla))} \vla$. In particular,
$J_{\max} \supset J(\vla) \setminus \{ i_l \}$ for $l=1,2$, whence
$J_{\max} = J(\vla)$. Thus $\wt_{I \setminus \{ i_l \}} \vla =
\wt_{J_{\max} \sqcup (I \setminus J(\vla))} \vla = \wt \vla$, so that $(I
\setminus \{ i_l \})_{\max} = J(\vla)$. The second assertion now follows
from the preceding paragraph.
\end{proof}

We can now prove the above characterization of facets in highest weight
modules.

\begin{proof}[Proof of Proposition \ref{Pfacet}]
We first show that $(1) \implies (2) \implies (6) \implies (5) \implies
(4) \implies (1)$.
Recall from Theorem \ref{Tfinite}(1) that the affine hull of $\wt \vla$
is $\lambda - \R \Delta_{I_{\min} \sqcup (I \setminus J(\vla))}$. Now if
(1) holds, then note by Proposition \ref{Phalfspace} that $i \notin
J(\vla) \setminus I_{\min}$, and also that if $i' \in I_{\min} \sqcup (I
\setminus J(\vla))$ then $F_{i'}(\vla) \subsetneq \wt \vla$. Thus if
$F_i(\vla) \subset F_{i'}(\vla)$ then they are equal.
Intersecting with $\wt \vla$ yields: $\wt_{I \setminus \{ i \}} \vla =
\wt_{I \setminus \{ i' \}} \vla$, which contradicts Proposition
\ref{Phalfspace} if $i' \neq i$. Thus $(1) \implies (2)$.
Now assume that (6) fails; then $\wt_{I \setminus \{ i \}} \vla$ does not
contain a nontrivial $\alpha_{i'}$-string or an $\alpha_i$-string.
Therefore
\[ \wt_{I \setminus \{ i \}} \vla = \wt_{I \setminus \{ i, i' \}} \vla
\subset \wt_{I \setminus \{ i' \}} \vla \]

\noindent by Lemma \ref{Lweights}, and the inclusion is strict by
Proposition \ref{Phalfspace}.
This contradicts (2), whence $(2) \implies (6)$. Clearly $(6) \implies
(5)$ since $\wt \vla \subset \lambda - \Z_+ \Delta$. We show next that
$(5) \implies (4)$. First observe from (5) that every $i' \in I_{\min}
\setminus \{ i \}$ is continued in $(I \setminus \{ i \})_{\min}$. It
remains to show the reverse inclusion that $(I \setminus \{ i \})_{\min}
\subset I_{\min} \setminus \{ i \}$. To show this, note by (5) that there
are two cases: first if $i \in I \setminus J(\vla)$, then $(I \setminus
\{ i \})_{\min} \subset I_{\min} = I_{\min} \setminus \{ i \}$, as
desired.
The other case is if $i \in I_{\min}$. Now compute using Theorem
\ref{T5}:
\begin{align*}
&\ \wt_{(I \setminus \{ i \})_{\min} \sqcup (I \setminus (\{ i \} \cup
J(\vla)))} \vla = \wt_{I \setminus \{ i \}} \vla = \wt_{I \setminus \{ i
\}} \vla \cap \wt \vla\\
= &\ \wt_{(J(\vla) \setminus \{ i \}) \sqcup (I \setminus (\{ i \} \cup
J(\vla)))} \vla \cap \wt_{I_{\min} \sqcup (I \setminus J(\vla))} \vla =
\wt_{(I_{\min} \setminus \{ i \}) \sqcup (I \setminus J(\vla))} \vla.
\end{align*}

\noindent Again using Theorem \ref{T5}, it follows that $(I \setminus \{
i \})_{\min} \subset I_{\min} \setminus \{ i \}$. This shows that $(5)
\implies (4)$. Now if (4) holds, then compute using Theorem
\ref{Tfinite}(1):
\begin{align*}
\dim F_i(\vla) = &\ |(I \setminus \{ i \})_{\min} \sqcup ((I \setminus \{ i
\}) \setminus J(\vla))|\\
= &\ |(J(\vla) \setminus \{ i \}) \sqcup ((I \setminus \{ i \}) \setminus
J(\vla))| = |I \setminus \{ i \}|,
\end{align*}

\noindent where the last equality follows from (4). Thus $(4) \implies
(1)$.

It remains to show that (3) is equivalent to the other assertions. Via
Lemma \ref{Lfacts} we will identify $\wt
L_{J(\vla)}(\pi_{J(\vla)}(\lambda - \mu))$ with $\wt L_{J(\vla)}(\lambda
- \mu)$. We now show that $(4) \implies (3) \implies (1)$. First recall
from above that $(I \setminus \{ i \})_{\max} = I_{\max} \setminus \{ i
\}$ is equivalent to $i \notin J(\vla) \setminus I_{\min}$. Now suppose
(4) holds, i.e., $(I \setminus \{ i \})_{\min} = I_{\min} \setminus \{ i
\}$. If $i' \in I_{\min} \setminus i$, then there exists a connected
component $C$ of the Dynkin diagram of $I_{\min} \setminus \{ i \}$ such
that $i' \in C$. By equation \eqref{Emin}, if $\pi_C(\lambda) = 0 \neq
(\Delta_C, \Delta_{I \setminus J(\vla)})$, then set $\mu := \alpha_{i_2}$
for $i_2 \in I \setminus J(\vla)$ such that $(\Delta_C, \alpha_{i_2})
\neq 0$; while if $\pi_C(\lambda) = 0$ then choose $\mu := 0$. It follows
that $i' \in C \subset K_3(\lambda - \mu)$, where $K = I_{\min} \setminus
\{ i \}$. Now apply the minimality of $I_{\min} \setminus \{ i \}$ by
(4), as well as \cite[Proposition 5.1]{KR}, to the finite-dimensional
module $L_{J(\vla)}(\lambda - \mu)$. Thus, there is at least one weight
$\mu'_i$ such that $\lambda - \mu'$ is a sum of simple roots with at
least one simple root equal to $\alpha_{i'}$. In particular, for each $i'
\in I_{\min} \setminus \{ i \}$ it follows that $(\mu_i, \omega_{i'})
\neq (\lambda, \omega_{i'})$, proving (3).

Finally, suppose (3) holds. Note that the affine hull of $\conv_\R \wt_{I
\setminus \{ i \}} \vla$ contains $\Delta_{I \setminus (\{ i \} \cup
J(\vla))}$ by definition of $J(\vla)$. Next, given $i' \in I_{\min}
\setminus \{ i \}$, it follows, using the notation of (3), that
$\wt_{I_{\min} \setminus \{ i \}} L_{J(\vla)}(\lambda - \mu) \subset \wt
\vla$. It follows by (3) that $\Delta_{I_{\min} \setminus \{ i \}}$ is
also contained in the affine hull of $\conv_\R \wt_{I \setminus \{ i \}}
\vla$. Therefore $(3) \implies (1)$ and the proof is complete.
\end{proof}

Finally, we use the analysis in this and previous sections to prove our
last main result.

\begin{proof}[Proof of Theorem \ref{Thalfspace}]
Since $\wt \vla \subset \lambda - \Z_+ \Delta$ and since $\wt \vla$ is
$W_{J(\vla)}$-stable by Theorem \ref{T1}, hence $\conv_\R \wt \vla
\subset \bigcap_{i \in I, w \in W_{J(\vla)}} H_{i,w}$. Also note that
$W_{(I \setminus \{ i \})_{\max}}$ preserves the half-space $H_{i,1}$ as
well as its boundary, since $\wt_{I \setminus \{ i \}} \vla$ is contained
in the boundary. It follows that the above intersection remains unchanged
even if it runs only over $\{ W^i : i \in I \}$.

Now since it is also known that $\conv_\R \wt \vla$ is a convex
polyhedron, it equals its minimal half-space representation, i.e., the
intersection of the half-spaces $H_{i,w}$ corresponding to the
codimension-one facets of $\conv_\R \wt \vla$. Note here that the
codimension of the faces of $\conv_\R \wt \vla$ is computed inside its
affine hull, which is $\lambda - \R \Delta_{I_{\min} \sqcup (I \setminus
J(\vla))}$ by Theorem \ref{Tfinite}(1).
By Theorem \ref{T23} for $\vla = M(\lambda,J(\vla))$, the codimension-one
faces of $\conv_\R \wt \vla$ correspond to $i \in I$ (and any $w$) such
that the supporting hyperplane of $H_{i,1}$, which is the affine hull of
$\conv_\R \wt_{I \setminus \{ i \}} \vla$, has codimension one. By
Proposition \ref{Pfacet}, this condition is equivalent to: $i \in
I_{\min} \sqcup (I \setminus J(\vla))$, and $(I \setminus \{ i \})_{\min}
= I_{\min} \setminus \{ i \}$.
Thus the second part of the assertion is proved, and hence the first part
as well.
\end{proof}
%}}}

%{{{1 Section 7 - Minimum elements and longest weights in compact faces
\section{Minimum elements and longest weights in compact faces}\label{S7}

In this section we present additional results on minimum elements in
``compact'' faces (i.e., ones containing only finitely many weights) of
$\conv_\R \wt \vla$. We also show that there is a natural analogue in any
highest weight module $\vla$, of the long roots in the adjoint
representation. The following result characterizes the standard parabolic
subsets of highest weight modules that contain minimal elements, and also
identifies these elements as well as the longest weights.

\begin{prop}\label{Pmin}
Fix $\lambda \in \lie{h}^*, M(\lambda) \twoheadrightarrow \vla$, and $J
\subset I$. Then the following are equivalent.
\begin{enumerate}
\item $\wt_J \vla$ has a longest element (i.e., a weight $\mu$ with
maximum Euclidean norm $(\mu,\mu)$).

\item $\conv_\R \wt_J \vla$ is a convex polytope.

\item $J \subset J(\vla)$.

\item $\wt_J \vla$ has a minimum element in the standard partial order on
$\lie{h}^*$.
\end{enumerate}

\noindent If these conditions hold, then the longest weights in $\wt_J
\vla$ are precisely $W_J(\lambda)$. These include the maximum and minimum
elements in $\wt_J \vla$ in the standard partial order, which are unique
and equal $\lambda$ and $w_\circ^J(\lambda)$ respectively.
\end{prop}

\begin{proof}
If (1) holds, then $\wt_J \vla$ cannot contain any string of the form
$\lambda - \Z_+ \alpha_i$ for $i \in I \setminus J(\vla)$, so (3)
follows. Clearly $(3) \implies (2)$ using Theorem \ref{T1}. Now if (2)
holds then $\wt_J \vla$ is a finite set, hence has a longest element.
Thus (1), (2), and (3) are equivalent. Next if $j \in J \setminus
J(\vla)$ then $\lambda - \Z_+ \alpha_j \subset \wt_J \vla$. Therefore
$(4) \implies (3)$. Conversely, if $J \subset J(\vla)$ then $\varpi_J :
\wt_J \vla \to \wt L_J(\pi_J(\lambda))$ is a bijection by Lemma
\ref{Lfacts}. Therefore it has a minimum element $w_\circ^J(\lambda)$,
which is unique by lowest weight theory. Hence (1)--(4) are equivalent.

Finally, note that every standard parabolic subset $\wt_J \vla$ always
contains the unique maximum element $\lambda$. Now if $J \subset I$ is
such that $\conv_\R \wt_J \vla$ is a convex, compact polytope, the norm
function attains its maximum value on this polytope. It is easy to verify
that the norm cannot be maximized at an interior point of a line segment.
Therefore the maximum value is attained at a vertex, i.e., at a point in
$W_{J \cap J(\vla)}(\lambda)= W_J(\lambda)$ (by Proposition \ref{P11}).
The proof is completed by recalling that $W_{J(\vla)} \subset W$ acts on
$\lie{h}^*$ by isometries.
\end{proof}
\pagebreak

It is also possible to obtain characterizations of the sets $J_3(\vla)$,
as well as of the minimal weights in compact faces of highest weight
modules. The following result, together with Proposition \ref{Pmin},
accomplishes these goals.

\begin{prop}\label{Pcompact}
Fix $\lambda \in \lie{h}^*$ and $M(\lambda) \twoheadrightarrow \vla$.
Given $\mu \in \wt M(\lambda)$, define
\begin{equation}
I_\lambda(\mu) := \{ i \in I : (\lambda - \mu, \omega_i) = 0 \}.
\end{equation}
\begin{enumerate}
\item Given $J \subset I$, $J = J_3(\vla)$ if and only if $J = J(\vla)
\setminus I_\lambda(\mu)$ for some $\mu \in \wt_{J(\vla)} \vla$.

\item Suppose $\mu \in \wt_{J(\vla)} \vla$. Then $\mu = \min
(\wt_{J(\vla) \setminus I_\lambda(\mu)} \vla)$ if and only if $\mu$ is a
longest weight in $\wt_{J(\vla)} \vla$ and $(\mu,\alpha_i) \leq 0$ for
all $i \in J(\vla) \setminus I_\lambda(\mu)$.
\end{enumerate}
\end{prop}

In particular, the result provides a characterization of the sets
$J_{\min}(\vla)$ for all finite-dimensional modules $\vla = M(\lambda,I)
= L(\lambda)$ when $\lambda \in P^+$.

Note that the $W_{J(\vla)}$-orbit of all ``longest weights'' in
$\wt_{J(\vla)} \vla$ is a generalization in arbitrary $\vla$ of the
$W$-orbit of long roots in the adjoint representation. Apart from the
length, the longest weights generalize to $\vla$ several properties
satisfied by the long roots in $L(\theta) = \lie{g}$ for simple
$\lie{g}$. For instance, Propositions \ref{Pmin} and \ref{Pcompact}
extend \cite[Proposition 3.3(1), Remark 3.8, Lemma 3.12, and Propositions
3.9 and 3.16]{CM} to all highest weight modules $\vla$.

\begin{proof}
(1) It is not hard to compute from the definitions that $\lambda -
w_\circ^J(\lambda) =$\break
$\sum_{j \in J_3(\vla)} c_j \alpha_j$ with all $c_j > 0$. The formula
\eqref{Emin} for $J = J_3(\vla)$ follows by using $\mu =
w_\circ^J(\lambda)$ and the definition of $I_\lambda(\mu)$.
Conversely, suppose $J = J(\vla) \setminus I_\lambda(\mu)$ for some $\mu
\in \wt_{J(\vla)} \vla$. Note by the definition of $I_\lambda(\mu)$ that
$\mu \in \wt_J \vla$ but $\mu \notin \wt_{J'} \vla$ for any $J'
\subsetneq J$. This minimality of $J$ implies by Theorem \ref{T5} that $J
= J_3(\vla)$.

(2) Set $J'_\mu := J(\vla) \setminus I_\lambda(\mu)$ for ease of
exposition. If $\mu = \min (\wt_{J'_\mu} \vla)$, then $\mu \in
W_{J(\vla)}(\lambda)$, whence $\mu$ is a longest weight. Moreover, $\mu$
is the lowest weight of the $\lie{g}_{J'_\mu}$-module
$L_{\lie{g}_{J'_\mu}}(\lambda)$, so $s_i(\mu) \geq \mu$ for $i \in
J'_\mu$.
This implies that $(\mu, \alpha_i) \leq 0$ for $i \in J'_\mu$. To prove
the converse, first note that since $\mu$ is a longest weight, it is of
the form $W_{J(\vla)}(\lambda)$. Moreover, $\mu \in \wt_{J'_\mu} \vla$
(as in the previous part) by definition of $I_\lambda(\mu)$. Therefore
$\mu \in W_{J(\vla)}(\lambda) \cap \wt_{J'_\mu} \vla =
W_{J'_\mu}(\lambda)$. The remainder of the argument follows the proof of
\cite[Proposition 3.16]{CM}.
\end{proof}
%}}}

%{{{1 Section 8 - Example: finite-dimensional representations over a simple Lie algebra
\section{Example: finite-dimensional representations over a simple Lie
algebra}\label{Scellini}

For completeness, we conclude this paper by pointing out several
connections between the results proved in this paper, and results of the
recent paper \cite{CM} by Cellini and Marietti as well as the subsequent
preprint \cite{LCL} by Li--Cao--Li. Throughout this section, let
$\lie{g}$ denote a complex simple Lie algebra.

We first discuss the work \cite{CM}, in which the authors study the faces
of the root polytope $\calp(\theta) := \conv_\R \Phi$ for $\lie{g}$. This
corresponds to the special case where $\lambda = \theta$ is the highest
root and $\vla = M(\theta,I) = L(\theta) = \lie{g}$. It is not hard to
see in this case that
\begin{equation}\label{Esimple}
J_\lambda = J(\vla) = I_{\max} = I = I_{\min}
\end{equation}

\noindent (since $\theta \neq 0$) and that $F_J$ equals $\conv_\R \wt_{I
\setminus J} L(\theta) = \conv_\R \wt_{I \setminus J} \lie{g}$ in our
notation. Now given $J \subset I$, \cite[Proposition 3.7]{CM} says that
$F_{J'} = F_J$ if and only if $\partial J \subset J' \subset
\overline{J}$. On the other hand, Theorem \ref{T5} implies in the special
case $\vla = M(\theta,I) = \lie{g}$ that
\[
\wt_{J'} \lie{g} = \wt_J \lie{g} \quad \Longleftrightarrow \quad
J_{\min} \subset J' \subset J_{\max}.
\]

\noindent Given the uniqueness of the sets $\partial J, \overline{J},
J_{\min}, J_{\max} \subset I$ for every $J \subset I$, it is now possible
to provide a dictionary between our notation and that used in \cite{CM}.

\begin{prop}\label{Pcellini}
Suppose $\vla = \lie{g}$ with $\lie{g}$ a simple Lie algebra, and $J
\subset I$ is arbitrary. Then,
\begin{equation}\label{Eminmax3}
\partial J = I \setminus (I \setminus J)_{\max}, \qquad
\overline{J} = I \setminus (I \setminus J)_{\min},
\end{equation}

\noindent or conversely, $J_{\min} = I \setminus \overline{(I \setminus
J)}$ and $J_{\max} = I \setminus \partial(I \setminus J)$.
\end{prop}

Note that it is not hard to formulate equation \eqref{Eminmax3}, by
comparing Proposition \ref{Pfacet}(4) and \cite[Theorem 4.5(4)]{CM}, as
well as from the formula for $(I \setminus \{ i \})_{\max}$ in
Proposition \ref{Phalfspace}.

\begin{proof}
The second assertion follows from equation \eqref{Eminmax3}. To prove the
first assertion, we first recall the definition of $\partial J$ and
$\overline{J}$ from \cite[Section 3]{CM}. Let $\widehat{I} := I \sqcup \{
0 \}$ correspond to the simple roots $\widehat{\Delta} := \Delta \sqcup
\{ \alpha_0 \}$ associated to the affine Lie algebra $\widehat{\lie{g}}$.
Now given $J \subset I$, let $(\widehat{I} \setminus J)_0$ denote the
connected component of the Dynkin diagram of $\widehat{I} \setminus J$
that contains the affine root $\alpha_0$, and define:
\begin{equation}\label{Ecellini}
\overline{J} := I \setminus (\widehat{I} \setminus J)_0, \qquad \partial
J := J \setminus (\widehat{I} \setminus J)_0^\perp.
\end{equation}

\noindent To see why (the first formula in) \eqref{Eminmax3} follows from
equations \eqref{Emin}, \eqref{Emax}, and \eqref{Ecellini}, first recall
from \cite[Chapter VI.4.3]{Bou} that in the Dynkin diagram for
$\widehat{\lie{g}}$, the affine root $\alpha_0$ can be viewed as the
negative of the highest root for simple $\lie{g}$: $\alpha_0 = -\theta$.
Thus the connected components $C$ of (the Dynkin diagram of) $I \setminus
J$ such that $\pi_C(\theta) \neq 0$ correspond precisely to those simple
roots $\alpha_i \in \Delta_{I \setminus J}$ such that $\alpha_0(h_i) =
-\theta(h_i) \neq 0$, i.e., such that $C$ is contained in the connected
component of $\widehat{I} \setminus J$ containing $\alpha_0$. Therefore
$(I \setminus J)_3(L(\theta)) \sqcup \{ 0 \} =
\{ \alpha_0 \} \sqcup \Delta_{(I \setminus J)_3(\lie{g})} = (\widehat{I}
\setminus J)_0$ with a slight abuse of notation.
Moreover, equations \eqref{Esimple} and \eqref{Emin} imply that $K_{\min}
= K_3(\lie{g})$ for all $K \subset I$. Therefore,
\[
I \setminus (I \setminus J)_{\min} = I \setminus (I \setminus
J)_3(\lie{g}) = I \setminus (\widehat{I} \setminus J)_0 = \overline{J}.
\]

\noindent This proves the first formula in \eqref{Eminmax3}. To show the
second, we study equation \eqref{Emax} in closer detail in the current
special case. By equation \eqref{Esimple}, $K_2(\lie{g})$ and
$K_4(\lie{g})$ are empty for all $K \subset I$, whence
$K_2(\lie{g})^\perp = I$. Therefore using equation \eqref{Emax} and the
above analysis for $\overline{J}$,
\begin{align*}
I \setminus (I \setminus J)_{\max} = &\ I \setminus \left( 
(I \setminus J) \sqcup (J \cap \{ \theta \}^\perp \cap (I
\setminus J)_{\min}^\perp \cap I) \right)\\
= &\ J \setminus (\{ \theta \}^\perp \cap (I \setminus
J)_3(\lie{g})^\perp)\\
= &\ J \setminus (\{ \alpha_0 \} \sqcup \Delta_{(I \setminus
J)_3(\lie{g})})^\perp = J \setminus (\widehat{I} \setminus J)_0^\perp =
\partial J,
\end{align*}

\noindent which proves the second formula in \eqref{Eminmax3}.
\end{proof}

\begin{remark}
The dictionary of Proposition \ref{Pcellini} immediately helps translate
the results in this paper, in the special case $\lambda = \theta$ and
$\vla = L(\theta) = \lie{g}$, into many results in \cite{CM}.
In particular, it follows that most of the results in Sections 1, 3, 4,
and 5 of \cite{CM} are specific manifestations of
representation-theoretic phenomena that occur for all highest weight
modules over all semisimple Lie algebras.
Moreover, in \cite{CM} the authors worked with the root system -- i.e.,
the adjoint representation -- and hence were able to prove their results
using purely combinatorial arguments. In contrast, because we work with
arbitrary highest weight modules, the present paper provides an
alternative, representation-theoretic approach to proving the results in
\cite{CM} for $L(\theta) = \lie{g}$.

A further addition to the dictionary of Proposition \ref{Pcellini}
involves observing that many of the results in \cite{CM} are stated in
terms of the affine root $\alpha_0$ and the affine root system
$\widehat{\Phi}$ corresponding to the simple Lie algebra $\lie{g}$. As
noted above, $\alpha_0$ is simply the negative of the highest root
$\theta$ of $\lie{g}$, i.e., the highest weight of the adjoint
representation.
\end{remark}

There are also certain results in \cite{CM} that do not hold for all
modules $\vla$, but are specific to the combinatorics of the root system,
i.e., the adjoint representation $\lie{g} = L(\theta)$. For instance, the
size of the set $V_J = \wt_{I \setminus J} L(\theta)$ is a computation
specific to the root system and is expressed in terms of other root
systems; see \cite[Theorem 1.1(1)]{CM}. However, there are other
statements that are specific to the adjoint representation and yet can be
obtained from our results in previous sections via Proposition
\ref{Pcellini}. We now provide an alternate proof of one such statement
from \cite{CM}.

\begin{cor}\label{Ccellini}
Suppose $\lie{g}$ is simple with highest root $\theta$, $\vla = \lie{g} =
L(\theta)$, and $i \in I \supset J$.
Then the coordinate face $F_i = \conv_\R \wt_{I \setminus \{ i \}}
\lie{g}$ is a codimension-one facet of the root polytope $\conv_\R \wt
\lie{g}$, if and only if $\widehat{I} \setminus \{ \alpha_i \}$ is
connected (i.e., the corresponding parabolic root subsystem is
irreducible).
\end{cor}

Note that the result is precisely \cite[Theorem 4.5 (1) $\Leftrightarrow$
(2)]{CM}. We now show that it quickly follows from the analysis in
previous sections, via Proposition \ref{Pcellini}.

\begin{proof}
Note by Proposition \ref{Phalfspace} that $(I \setminus \{ i \})_{\max} =
I \setminus \{ i \}$ for all $i \in I$. Now apply the equivalence $(1)
\Longleftrightarrow (4)$ of Proposition \ref{Pfacet} for $\vla = \lie{g}$
simple). Thus, $F_i$ is a codimension-one facet if and only if $(I
\setminus \{ i \})_{\min} = I \setminus \{ i \}$, if and only if
$\overline{\{ i \}} = \{ i \}$ by Proposition \ref{Pcellini}. It is easy
to see using equation \eqref{Ecellini} that this condition is equivalent
to $\widehat{I} \setminus \{ \alpha_i \}$ being connected.
\end{proof}

\begin{remark}\label{Rcellini}
We make two further observations related to the recent paper \cite{CM}.
Recall that Proposition \ref{Pvinberg} provided a complete description of
the inclusion relations among the sets $\wt_J \vla$ (or their convex
hulls) for any highest weight module $\vla$:
\begin{align*}
\wt_J \vla \subset \wt_{J'} \vla \quad \Longleftrightarrow \quad &
J_{\min} \subset J'_{\min}, \ J \setminus J(\vla) \subset J' \setminus
J(\vla)\\
\Longleftrightarrow \quad &
J_{\max} \subset J'_{\max}, \ J \setminus J(\vla) \subset J' \setminus
J(\vla).
\end{align*}

\noindent Such a formulation (via the dictionary mentioned in Proposition
\ref{Pcellini}) was also discussed in the special case of the
finite-dimensional adjoint representation for simple $\lie{g}$ in
\cite[Remark 4.6]{CM}.

Next, for completeness we recall an interesting result shown recently by
Cellini and Marietti for the adjoint representation of a simple Lie
algebra. Namely, the authors show in \cite[Theorem 5.2]{CM} that no
standard parabolic subset of $\wt \lie{g}$ is the union of two nontrivial
orthogonal subsets. It is natural to ask if this result holds for other
highest weight modules. However, this is not the case; in fact the result
fails to hold even if $\lie{g}$ is simple and $\vla$ is
finite-dimensional. For example, when $\lie{g}$ is of type $C_2$ and
$\lambda = \theta_s$, the highest short root, the (nonzero) weights of
$L(\theta_s)$ comprise a root system of type $A_1 \times A_1$, hence can
be partitioned into two nontrivial orthogonal subsets.
\end{remark}

We end by pointing out that several of the main results in \cite{CM} hold
not only for the root polytope, but also for a large family of Weyl
polytopes:

\begin{prop}\label{Pcm}
(Notation as in Theorem \ref{Tcm1} and Definition \ref{Dminmax}.)
Suppose $\supp(\lambda) = \supp(\theta)$ for some $\lambda \in P^+$.
Define $F_J := \conv_\R \wt_{I \setminus J} L(\lambda)$ for $J \subset
I$. Then \cite[Theorems 1.2, 1.3]{CM} hold for $\conv_\R(\wt
L(\lambda))$, with the exact same formulas.
\end{prop}

One can similarly show that many of the other results in \cite{CM} also
go through completely unchanged for $\conv_\R(\wt L(\lambda))$, if
$\lambda$ and $\theta$ have the same support among the fundamental
weights $\Omega$.

\begin{proof}
The result follows from Theorem \ref{Tfinite}, Lemma \ref{Lfinite},
Proposition \ref{Pcellini}, and Corollary \ref{Ccellini} if we show that
equation \eqref{Eminmax3} holds for all $J \subset I$, for the simple
finite-dimensional module $\vla = L(\lambda)$. But this is clear by
Theorem \ref{T5}: the formulas in equations \eqref{Emin} and \eqref{Emax}
for $J_{\min}, J_{\max}$ only depend on $\supp(\lambda)$ as well as $J$
and $J \setminus J(\vla) = \emptyset$.
\end{proof}

\subsection{Finite-dimensional modules}

Since the present paper was submitted, the work \cite{LCL} by Li--Cao--Li
has subsequently appeared. In it, the authors have independently extended
Cellini--Marietti's results in \cite{CM} (by different methods than our
analysis in the present paper) to all finite-dimensional simple modules
over a simple Lie algebra $\lie{g}$. As our analysis holds more generally
for all highest weight modules and over all semisimple $\lie{g}$, we
conclude this section by discussing the main results of \cite{LCL}, and
how they fit into our framework in a manner similar to the results in
\cite{CM}.

In order to make the subsequent discussion more consistent with the
analysis both in \cite{CM} as well as above in this section, we begin by
setting some notation. Fix a simple Lie algebra $\lie{g}$ and a dominant
integral weight $\lambda \in P^+ \setminus \{ 0 \}$. In \cite{LCL}, the
authors define the \textit{extended Coxeter diagram} by adding a new node
$-\lambda$ to the Dynkin diagram of $\lie{g}$, with a (single additional)
edge between $-\lambda$ and a previous node $i \in I$ if $(\lambda,
\alpha_i) > 0$. Now set $\widehat{I}_\lambda := I \cup \{ -\lambda \}$,
and for $J \subset I$, define $(\widehat{I}_\lambda \setminus J)_0$ to be
the connected component of the extended Dynkin diagram of
$\widehat{I}_\lambda \setminus J$ that contains the new node $-\lambda$.
Finally, define $\overline{J} := I \setminus (\widehat{I}_\lambda
\setminus J)_0$.

In \cite{LCL}, the authors study the Weyl polytope $\calp(\lambda) =
\conv_\R \wt L(\lambda)$, focusing on its ``standard parabolic faces''
$F_J^\lambda := \conv_\R \wt_{I \setminus J} L(\lambda)$. In the above
notation, their main results are as follows:
\begin{itemize}
\item There exists a bijection between the distinct elements in the
multiset\break
$\{ \conv_\R \wt_J L(\lambda) : J \subset I \}$ and the connected
sub-diagrams of the extended Coxeter diagram on $\widehat{I_\lambda}$
that contain the node $-\lambda$. The authors observe in \cite{LCL} that
this bijection is a special case of the \textit{Putcha--Renner Recipe}
\cite[Theorem 4.16]{PuRe}, which yields connections to the theory of
algebraic monoids.

\item Each face $F_J^\lambda$ has affine hull $\displaystyle \R \Delta_{I
\setminus \overline{J}} = \R \Delta_{(\widehat{I}_\lambda \setminus J)_0
\setminus \{ -\lambda \}}$, which is spanned by roots. Moreover, $\dim
F_J^\lambda = | (\widehat{I}_\lambda \setminus J)_0 | - 1$.
\vfill\pagebreak

\item Given $J,J' \subset I$, $F_J^\lambda = F_{J'}^\lambda$ if and only
if $(\widehat{I}_\lambda \setminus J)_0 = (\widehat{I}_\lambda \setminus
J')_0$, or equivalently, if $\overline{J} = \overline{J'}$.

\item The barycenter of the face $F_J^\lambda$ is a nonnegative rational
linear combination of the fundamental weights in $\Omega_J$.

\item The $f$-polynomial of the Weyl polytope $\conv_\R \wt L(\lambda)$
is
\[ {\bf f}_{L(\lambda)}(t) = \sum_J [W : W_{(I \setminus J) \cup
(\widehat{I}_\lambda \setminus J)^\perp}] t^{|\widehat{I}_\lambda
\setminus J|-1}, \]

\noindent where the sum runs over all $J \subset I$ such that
$\widehat{I}_\lambda \setminus J$ is a connected sub-diagram of the
extended Coxeter diagram.
\end{itemize}

Note that when $\lambda = \theta$ is the highest root of $\lie{g}$, the
corresponding extended Coxeter diagram on $\widehat{I}_\theta$ is
precisely (the simple graph underlying) the affine Dynkin diagram
corresponding to $\lie{g}$, and in this situation the above results were
shown in \cite{CM}. Similarly, the case where $\lambda \in P^+$ has the
same support as $\theta$ was previously worked out in Proposition
\ref{Pcm}. For other $\lambda \in P^+$, the results in \cite{LCL} are
formulated along the lines of those in \cite{CM}, and extend many of the
results in \cite{CM}.

We now explain why for all $\lambda \in P^+$, the results shown in
\cite{LCL} also follow from the above analysis in the present paper, if
the same dictionary as in Proposition \ref{Pcellini} is used. To do so,
one first needs to define the analogue of the set $\partial J$ of simple
roots in \cite{CM}, which the authors do not define for general $\lambda
\in P^+$ in \cite{LCL}. Thus, we define (via Definition \ref{Dminmax}(5))
\begin{equation}
\partial J := J \setminus (\widehat{I}_\lambda \setminus J)_0^\perp
\subset I.
\end{equation}

Now observe that Proposition \ref{Pcellini} (with the new definitions of
$\overline{J}$ and $\partial J$) holds for $\vla = L(\lambda)$ for all
$\lambda \in P^+$, with the proof essentially unchanged. Moreover, the
formulas for $J_{\min}, J_{\max}$ depend only on $\lambda$ through its
support, since one still has $J(\vla) = J(L(\lambda)) = I$. It follows
that the results in \cite{LCL} (and others, including analogues of the
results in \cite{CM} and our main results in Section \ref{S3}) hold in
$\wt L(\lambda)$ for all $\lambda \in P^+$.
%}}}

\section*{Acknowledgment}

The author would like to thank the anonymous referee for providing
numerous useful comments and suggestions, which helped improve the paper.

%{{{1 Bibliography

%}}}

\bigskip

\end{document}